\documentclass[12pt,a4paper,reqno]{amsart}
\usepackage{amsfonts}
\usepackage{amsthm}
\usepackage{amsmath}
\usepackage{amssymb}
\usepackage{amscd}
\usepackage{inputenc}
\usepackage{t1enc}
\usepackage[mathscr]{eucal}
\usepackage{indentfirst}
\usepackage{graphicx}
\usepackage{graphics}
\numberwithin{equation}{section}
\usepackage[margin=2.9cm]{geometry}
\usepackage{epstopdf} 
\usepackage{multirow}
\usepackage{mathtools}
\usepackage{caption}
\usepackage{setspace}
\usepackage[leqno]{amsmath}

\usepackage[style=alphabetic, backend=bibtex]{biblatex}
\newcommand{\bibhack}[1]{}

\theoremstyle{plain}
\newtheorem*{theorem*}{Theorem}
\newtheorem*{cor*}{Corollary}
\newtheorem*{lem*}{Lemma}
\newtheorem{Th}{Theorem}[section]
\newtheorem{Lemma}[Th]{Lemma}

\newtheorem{Cor}[Th]{Corollary}
\newtheorem{cor}[Th]{Corollary}

 \theoremstyle{definition}
\newtheorem{Def}[Th]{Definition}

\newtheorem{Rem}[Th]{Remark}
\newtheorem{Rmk}[Th]{Remark}
\newtheorem{DefLem}[Th]{Definition and Lemma}
\newtheorem{?}[Th]{Problem}
\newtheorem{Ex}[Th]{Example}

\newtheorem{maintheorem}{Theorem}[]
\newtheorem{maincor}[maintheorem]{Corollary}

\theoremstyle{plain}
\newtheorem{innercustomgeneric}{\customgenericname}
\providecommand{\customgenericname}{}
\newcommand{\newcustomtheorem}[2]{%
  \newenvironment{#1}[1]
  {%
   \renewcommand\customgenericname{#2}%
   \renewcommand\theinnercustomgeneric{##1}%
   \innercustomgeneric
  }
  {\endinnercustomgeneric}
}

\newcustomtheorem{customthm}{Theorem}
\newcustomtheorem{customtheorem}{Theorem}
\newcustomtheorem{customlemma}{Lemma}
\newcustomtheorem{customcor}{Corollary}

\newcommand{\Fix}{\operatorname{Fix}}
\newcommand{\codim}{\operatorname{codim}}

\newcommand{\Nu}{\mathcal{V}}

\newcommand{\Ztwo}{\mathbb{Z}_2}
\newcommand{\Zp}{\mathbb{Z}_p}

\newcommand{ \floor }[1]{ \lfloor #1 \rfloor }

\newcommand{\Q}{\mathbb{Q}}
\newcommand{\Z}{\mathbb{Z}}
\newcommand{\R}{\mathbb{R}}
\newcommand{\C}{\mathbb{C}}
\newcommand{\s}[1]{\mathbb{S}^{#1}}
\newcommand{\cp}{\mathbb{CP}}
\newcommand{\CP}[1]{\cp^{\frac{#1}{2}}}
\newcommand{\hp}{\mathbb{HP}}
\newcommand{\HP}[1]{\hp^{\frac{#1}{4}}}
\newcommand{\into}{\xhookrightarrow{}}
\newcommand{\acton}{\curvearrowright}
\newcommand{\fk}{\Fix\ker}
\newcommand{\cayley}{Ca\mathbb{P}^2}

\newcommand{\Hp}[2][*]{H^{#1}(#2, \Zp)}
\newcommand{\Htwo}[2][*]{H^{#1}(#2, \Ztwo)}
\newcommand{\Hz}[2][*]{H^{#1}(#2, \Z)}
\newcommand{\Hq}[2][*]{H^{#1}(#2, \Q)}

\newcommand{\tensor}{\otimes}

\begin{document}

\title[Improved four-periodicity and a conjecture of Hopf]{An improved four-periodicity Theorem and a conjecture of Hopf with symmetry}

\author[J. Nienhaus]{Jan Nienhaus}


\begin{abstract}
	In the 1930s, H. Hopf conjectured that a closed, even-dimensional manifold of positive sectional curvature has positive Euler characteristic.\\
	We show this under the additional assumption of an isometric $T^4$-action on the manifold, improving from previous theorems of Kennard, Wiemeler and Wilking assuming a $T^5$-action. More specifically, this is achieved by giving a rational cohomology classification of possible fixed point components.\\
	The main new tool is an improvement on the four-periodicity theorem originally developed by Kennard through the use of characteristic class theory.\\
	As a second application we give a rational cohomology classification of closed positively curved even-dimensional manifolds without odd rational cohomology that admit an isometric $T^6$-action.
\end{abstract}
\maketitle

\section{Introduction}
The study of manifolds admitting Riemannian metrics of positive or nonnegative sectional curvature has been of interest since the early days of the field of Riemannian geometry. Despite these many years of study, the only known way to differentiate the class of closed manifolds admitting positive sectional curvature within the ostensibly much larger class of those admitting nonnegative sectional curvature is a restriction on the possible fundamental groups; in particular, there is no known obstruction for nonnegatively curved simply connected closed manifolds for admitting positive sectional curvature.

This lack of obstruction is in stark contrast to the expectation one might have from considering the available examples. While there is an abundance of nonnegatively curved examples, like all compact Lie groups with their bi-invariant metrics and their quotients by closed subgroups, all exotic $7$-spheres and even the connected sums of two compact rank one symmetric spaces, the situation in positive curvature is remarkably different:

In even dimensions, the only known simply connected examples are spheres, projective spaces, the flag spaces of projective planes (due to Wallach \cite{wallach}) and the inhomogeneous biquotient $E^6=SU(3)//T^2$ (found by Eschenburg \cite{Eschenburg1982}). In particular, all known examples above dimension 24 are rank one symmetric spaces. The situation in odd dimension is not much better. While there are examples of low-dimensional families in dimensions 7 (Aloff-Wallach \cite{AloffWallach}, Eschenburg \cite{Eschenburg1982}) and 13 (Berger \cite{Berger}, Bazaikin \cite{Bazaikin}), starting from dimension 15 the only known examples are covered by spheres.\\

There are several conjectures about manifolds admitting positive curvature that are all difficult and open problems. Here, we will focus on the following, which remains unsolved after almost a century:

Back in the early 1930s, H. Hopf conjectured that any positively curved closed even-dimensional Riemannian manifold $M$ has positive Euler characteristic. 

In an attempt to bring new motion into the field of positive curvature, in the early 1990s Karsten Grove started what is now known as the `Grove symmetry program', whose aim is to study manifolds of positive curvature under additional assumptions of symmetry, the hope being that, after a good understanding has been achieved in some symmetric setting, new structure, obstructions, or new examples could be produced by successively weakening the symmetry assumptions.

One such symmetry assumption is requiring a lower bound on the \textit{symmetry rank} of $(M^n, g)$, i.e. on the rank of its isometry group.

The starting point here is a result of Grove and Searle \cite{grovesearle}: The symmetry rank $r$ of $M$ is at most $\floor{\frac{n+1}{2}}$, with equality only if $M$ is diffeomorphic to a sphere, real or complex projective space or lens space.\\

This linear bound could later be improved: Using his new Connectedness Lemma and an induction argument, Wilking \cite{w03} showed that if $M$ is simply connected and $r\ge \frac{n}{4}+1$ (and $n\ne 7$), $M$ is still homotopy equivalent to a sphere or complex or quaternionic projective space.\\

Moving from linear to sublinear bounds took another decade: In 2013, Kennard proved that $\chi(M)$ was positive if $n$ was a multiple of $4$ even if one only assumed $r\ge 2\log_2(n)$. This improvement was made using his four-periodicity theorem to obtain structure on relevant cohomology rings.\\

Moving to the present, in 2021 Kennard, Wiemeler and Wilking together showed that, in fact, one can improve even logarithmic bounds, when they proved $\chi(M)>0$ if $r\ge 5$ in all even dimensions. In particular, they showed

\begin{customthm}{}[\cite{kww}, Theorem A]
Let $(M^n, g)$ be a closed, orientable Riemannian manifold of positive sectional curvature equipped with an effective isometric $T^5$-action, and $F^f$ be a component of the fixed point set $Fix(T^5)$. \\
Then $H^*(F, \Q)$ is, as a ring, that of $\s{f}, \CP{f}$, or $\HP{f}$.
\end{customthm}

This implies the Hopf conjecture for these manifolds since, up to possibly passing to the oriented double cover, $M$ has the same Euler characteristic as the fixed point set $Fix(T^5)$ by a result of Lefschetz, and this fixed point set is non-empty due to a result of Berger.

We prove the following, improving the result from $T^5$ to $T^4$:

\begin{maintheorem}\label{mainfix}
Let $(M^n, g)$ be a closed, orientable Riemannian manifold of positive sectional curvature equipped with an effective isometric $T^4$-action, and $F^f$ be a component of the fixed point set $Fix(T^4)$. \\
Then $H^*(F, \Q)$ is, as a ring, that of $\s{f}, \CP{f}$, or $\HP{f}$.
\end{maintheorem}

This implies our contribution towards the Hopf conjecture, 

\begin{customcor}{A$'$}[Hopf-Conjecture with $T^4$-symmetry]\label{hopf}\,\\
Let $(M^n, g)$ be a closed, even-dimensional Riemannian manifold of positive sectional curvature.\\
If $M$ admits an effective isometric $T^4$-action, then $\chi(M)>0$.
\end{customcor}

In order to obtain these improvements, our main new tool will be improvements in the machinery of periodicity originally used by Kennard. In particular, we prove 

\begin{maintheorem}\label{mainper}
Let $(M^n, g)$ be a closed, orientable Riemannian manifold and $N^{n-k}\subset M$ an orientable submanifold with $dim(N)$-connected inclusion map. If
\begin{itemize}
\item $k\le \frac{n}{2}$ and the normal bundle $\Nu(N)$ admits a complex structure, or
\item $k\le \frac{n}{3}$,
\end{itemize}
then $\Hq{M}$ is 4-periodic.
\end{maintheorem}

This is an improvement of the original four-periodicity theorem of Lee Kennard, who showed this in the case $k\le \frac{n}{3}$ with $M$ simply connected.

\begin{customtheorem}{}[Four-periodicity Theorem, \cite{k13}]
Let $(M^n, g)$ be a closed, orientable, simply connected Riemannian manifold such that $\Hz{M}$ is $k$-periodic for $k\le \frac{n}{3}$.

Then $\Hq{M}$ is $4$-periodic.
\end{customtheorem}
(See Section 2 for the definition of periodicity; in particular, $\Hz{M}$ is $k$-periodic under the assumptions of theorem \ref{mainper})

Kennard's theorem appears to be sharp if one looks at the available examples: The Cayley projective plane $Ca\mathbb{P}^2$ is a 16-dimensional closed manifold with cohomology ring $\Z[x]/{x^3}$, where the generator $x$ is in degree 8, and contains a maximally connected $S^8$, which hits this generator. Since the Cayley plane is clearly not $4$-periodic rationally, improving the assumptions to only require $k\le \frac{n}{2}$ does not seem plausible in general. We are able to achieve this improvement anyways by noting that we mainly care about submanifolds given as the fixed point components of a circle action. In this case the normal bundle of $N$ comes naturally with the isotropy action of this circle, giving it the structure of a complex vector bundle, which turns out to be just enough additional structure for improvement.

Finally, we prove a variation of
\begin{theorem*}\cite[Theorem C, even-dim]{kww}
Let $M^n$ be a connected, closed, oriented even-dimensional manifold with $b_{odd}(M, \Q)=0$, admitting a smooth and effective $T^d$-action with $d\ge 4$ and at least one of the following:
\begin{itemize}
\item Every fixed point component of every codimension 3 subtorus is a rational cohomology $\s{m}, \cp^m$ or $\hp^m$.
\item Every fixed point component of every codimension 2 subtorus is a rational cohomology $\s{m}, \cp^m$ or $\hp^m$ and there is a n-dimensional rank one symmetric space $C$ with $\chi(M)=\chi(C)>0$.
\end{itemize}
Then $M$ has the rational cohomology of $\s{n}, \CP{n}, \HP{n}$ or $Ca\mathbb{P}^2$.
\end{theorem*}
We prove 
\begin{maintheorem}\label{maincomb}
Let $M^n$ be an even-dimensional closed oriented manifold with $b_{odd}(M)=0$, equipped with a $T^d$-action, $d\ge 5$, such that all fixed point components of codimension-2 subtori are of rational cohomology $\s{}, \cp$, or $\hp$ type.

Then $M$ has the rational cohomology of $\s{n}, \CP{n}$, or $\HP{n}$.
\end{maintheorem}

We achieve this result without reliance on any fixed point components of codimension 3 tori by replacing a key step in the argument using results of real projective combinatorics, allowing for a simplification of part of the original proof while at the same time showcasing a notable qualitative difference between actions of tori of dimensions $\le 4$ and $\ge 5$ in this setting. In exchange, we are no longer able to obtain the same rigidity for $T^4$-actions as the original theorem, see however Remark \ref{t4remark} on the $T^4$-case.\\

Combining Theorems \ref{mainfix} and \ref{maincomb}, we obtain the final main theorem of this thesis:

\begin{maincor}
Let $(M^n, g)$ be an even-dimensional closed Riemannian manifold of positive sectional curvature  with $b_{odd}(M, \Q)=0$, equipped with an effective isometric $T^6$-action.\\
Then $H^*(M, \Q)$ is, as a ring, that of $\s{n}, \CP{n}$, or $\HP{n}$.
\end{maincor}

In these theorems the perhaps inelegant requirement of vanishing odd rational cohomology can be understood as a technical issue. It is already a consequence of positive sectional curvature if one assumes the Bott-Grove-Halperin ellipticity conjecture \cite{grove}, as $\chi(M)>0$ by Corollary \ref{hopf}, see also the discussion at the beginning of Section \ref{C}.

\subsection{Structure}\,\\
In Section 2, we build up the toolbox of periodicity and prove Theorem \ref{mainper}. In Section 3, we use this to prove Theorem \ref{mainfix} and Corollary \ref{hopf}. We prove Theorem \ref{maincomb} in Section 4. Appendix A contains a section on calculations in $H^*(BU)$ to produce a self-contained proof of Lemma \ref{appalemma}. Appendix B contains the classification of $T^4$-representations without finite isotropy groups that produces Table \ref{tab:t4list}.

\subsection{Acknowledgements}\,

I would like to thank my advisor, Burkhard Wilking, for pointing me to these interesting problems, and Michael Wiemeler, for sharing with me his proof of Lemma \ref{wiemeler}. The author was supported by the Deutsche Forschungsgemeinschaft (DFG,
German Research Foundation) – Project-ID 427320536 – SFB 1442, as well as under Germany's Cluster of Excellence Strategy EXC 2044 - 390685587, Mathematics
M\"unster: Dynamics–Geometry–Structure.

\tableofcontents

\newpage
\section{An Improved four-periodicity Theorem}\label{A}

In this Section we prove

\begin{customtheorem}{\ref{mainper}}
Let $(M^n, g)$ be a closed, orientable Riemannian manifold and \\ $N^{n-k}\subset M$ an orientable submanifold with $dim(N)$-connected inclusion map. If \\
\begin{itemize}
\item $k\le \frac{n}{2}$ and the normal bundle $\Nu(N)$ admits a complex structure, or
\item $k\le \frac{n}{3}$,
\end{itemize}
then $\Hq{M}$ is 4-periodic.
\end{customtheorem}

In order to do this, we first build up the toolbox of periodicity, stating and proving in section 2.1 the basic lemmas in a sufficiently general way for later use. We then use these to show improved mod $p$ periodicity theorems for periodicity induced by complex characteristic classes, culminating in a proof of Theorem \ref{mainper} in section 2.2. For this, we use a result on the action of the Steenrod algebra on the complex classifying space $BU$. An elementary proof of this can be found in Appendix A.

Comparing the second case of theorem \ref{mainper} with the original four-periodicity theorem of Kennard, one notes that here the additional requirement of a submanifold/vector bundle inducing the periodicity replaces Kennard's requirement of vanishing fundamental group. In practice, where periodicity tends to arise from such submanifolds, this simplifies proofs slightly as it eliminates the need for covering arguments.

\subsection{Basic results on periodicity}

In this chapter we build up the theory of manifolds with periodic cohomology and give proofs for (sufficiently generalized versions of) those parts that will play a role in our proof of Theorem \ref{mainper}.\\

The starting point for four-periodicity in the context of positively curved manifolds is a pair of closed oriented manifolds $N^{n-k}\into M^{n}$ such that the inclusion is $(n-k)$-connected. In this setting there is already periodicity: in fact, by \cite{w03}, cup product with the Euler class $e(\Nu)$ of the normal bundle $\Nu$ of $N$ in $M$, given in the cohomology of $M$ either directly by $\iota_*[N]\cap[M]$ or more abstractly as the pullback of the Thom class under $(M, \emptyset)\to (M, M\backslash N) \approx_{H^*} (T(\Nu N), \infty)$, induces isomorphisms in all degrees $0<i<n-k$, and is injective in degree $i=n-k$, surjective in degree $i=0$.\\
To see that this is in fact true, we realize $\cup e(\Nu)$ as a composition of maps
\begin{align*}\label{perdiag}
H^*(M) \xrightarrow{\iota^*} H^*(N) \xrightarrow{} H_{n-k-*}(N) \xrightarrow{\iota_*} H_{n-k-*}(M) \xrightarrow{} H^{*+k}(M). \tag{{$\star$}}
\end{align*}
Note that $e(\Nu)$ becomes the `true' Euler class in $H^*(N, \Z)$ after restriction, justifying our naming.\\
The unlabeled maps are Poincar\'{e}-dualities $\cap [N], \cap [M]$ of $N$ and $M$, both isomorphisms in all degrees. The behaviour of $\cup e(\Nu)$ in extremal degrees is explained by the fact that $\iota_*$ is only surjective, $\iota^*$ only injective, in degree $n-k$.\\

\begin{Def}\label{perdef} We say that a class $x\in H^*(M^n)$ induces periodicity, if either 
\begin{itemize}
\item $|x|\le \frac{n}{2}$ and $\cup x:H^i(M)\to H^{i+|x|}(M)$ is injective for $0<i\le n-|x|$ and surjective for $0\le i < n-|x|$, or
\item $|x|\le n$ and $x$ is a product of such elements.
\end{itemize}

Note that in the second case, one still has that  $\cup x:H^i(M)\to H^{i+|x|}(M)$ is injective for $0<i\le n-|x|$ and surjective for $0\le i < n-|x|$.\\

We say $H^*(M)$ is $k$-periodic if there is such an element $x$ in degree $|x|=k$.
\end{Def}

\begin{Rem}Whenever we omit coefficients, we mean something to hold with an arbitrary, but fixed, coefficient ring, and specify it only in the situations where the choice matters.
\end{Rem}

\begin{Rmk} Earlier work usually places no restrictions on $|x|$ in the above definition, providing for a weaker condition for inducing periodicity for elements in high degrees. However, we exclude this on purpose in order to avoid ambiguity when talking about high degree elements inducing periodicity. For illustration, using the weaker definition, every generator of $H^{n-1}(M)$ would automatically induce periodicity whenever this group is singly generated, which would make several of the following lemmas false as stated.\\
\end{Rmk}

\begin{Rmk}
Positively curved manifolds are not the only context in which cohomology rings exhibit periodicity. One other notable appearance of this phenomena is in the study of finite groups, where sometimes one finds periodicity in the cohomology of the classifying space $BG$. For an overview on finite groups with periodic cohomology, see \cite{Wall2013} and the references therein. Note that if a space $BG$ satisfies Definition \ref{perdef} with $n=\infty$, the following lemmas and theorems remain true, setting  $n=\infty$ everywhere.
\end{Rmk}

Note that any product of elements satisfying the condition on cup products from Definition \ref{perdef} satisfies it again. One consequence of the degree requirement is that the converse is true as well:

\begin{Lemma}\label{ycupz}
If $x\in H^*(M)$ induces periodicity and $x=y\cup z$, then $y$ and $z$ also induce periodicity.
\end{Lemma}

\begin{proof}
We reduce first to the case where $|x|\le \frac{n}{2}$: If this is not the case, by definition we can write $x$ as product $x=x_1 x_2$ of lower degree elements inducing periodicity. We may assume $|x_1|\le |x_2|$. Then $x_1$ must have degree less than or equal to at least one of $y,z$, say $y=y_1 x_1$ (Note that, by definition, $\cup x_1$ is surjective into $H^{|y|}(M)$). This then gives an equation $x_2 = y_1 z$ and if we show that $y_1$ induces periodicity, $y$ induces periodicity by definition. We may assume $k:=|x|\le \frac{n}{2}$.\\

By symmetry of $y$ and $z$ it is enough to show the claim for $y$:\\
We give the argument for injectivity of $\cup y$, the one for surjectivity is analogous.\\
For $0\le i \le n-k$, we can write $\cup x: H^i \to H^{i+k}$ as a composition $\cup x=\cup z \circ \cup y$. As the first map in an injective composition, $\cup y: H^i \to H^{i+|y|}$ must be injective itself.\\
If $n-k < i \le n-|y|$, consider $\cup x \circ \cup y = \cup y \circ \cup x: H^{i-k}\to H^{i+|y|}$.\\
Since we already know $\cup y$ is injective in degree $i-k$ and $\cup x$ is injective in degree $i-k+|y|$ and an isomorphism in degree $i-k$, the claim follows from $\cup y = \cup x \circ \cup y \circ (\cup x)^{-1}$.\\
Note also that the bound $|x|\le \frac{n}{2}$ is sharp to ensure these two constructions cover all degrees.
\end{proof}

\begin{cor}\label{gcd}
Let $x\in H^k(M)$, $y\in H^l(M)$ induce periodicity in $H^*(M)$.\\
Then there is an element inducing periodicity in $H^*(M)$ in degree $\gcd(k,l)$.
\end{cor}
\begin{proof}
Let us assume $k>l$. By assumption, $\cup y$ is surjective into $H^k(M)$, so we may write $x=y\cup z$ for some $z\in H^{k-l}(M)$. By the previous Lemma, $z$ again induces periodicity. Replacing the pair $x,y$ with $y,z$ allows us to perform the Euclidean algorithm, terminating with an element in degree $\gcd(k,l)$. 
\end{proof}

Another property which we will use extensively is the following, which is a trivial but crucial consequence of the definition:

\begin{Lemma}\label{sums}
Let $F$ be a field and $x\in H^*(M, F)$ induce periodicity.
\begin{itemize}
\item If $y\in H^{|x|}(M, F)$ is nonzero, $y$ induces periodicity. 
\item If $x=y+z$, then at least one of $y,z$ induces periodicity.
\item If $y=x+z$, then at least one of $y,z$ induces periodicity.
\end{itemize}
\end{Lemma}
\begin{proof}
By definition, $\cup x: H^0(M, F)\to H^{|x|}(M, F)$ is surjective, so $H^{|x|}(M, F)$ is generated by $x$. Then, in the second case, at least one of $y$ and $z$ is a nonzero multiple of $x$, and in the first case $y$ is a nonzero multiple of $x$. The coefficient is invertible since $F$ is a field. The third statement is equivalent to the second.
\end{proof}
Earlier we have seen that a maximally connected inclusion of submanifolds induces a periodicity in integral cohomology of the ambient manifold $M$.\\
Replacing coefficient rings in the above diagram (\ref{perdiag}), we also see immediately that the restriction $e(\nu)_p$ to $H^*(M, \Zp)$ induces periodicity in $H^*(M, \Zp)$.\\
On the other hand, inducing periodicity in $H^*(M, \Zp)$ for some $p$ is not much worse than in integral cohomology if we are only interested in rational conclusions. In fact, the following Lemma holds:\\

\begin{Lemma}\label{ptoq}
If $H^*(M, \Z)$ is $k$-periodic with $2k\le n$ and if for some $z\in H^*(M, \Z)$ the mod $p$ restriction $z|_p$ induces periodicity in $H^*(M, \Zp)$, then $z|_\Q$ induces periodicity in $H^*(M, \Q)$.\\
\end{Lemma}
\begin{proof}\,\\
Let $x\in H^k(M, \Z)$ induce periodicity. After possibly dividing $z$ by $x$ (cf Lemma \ref{ycupz}), we may assume that the degree of $z$ is lower than that of $x$.\\
In the degenerate case $z_p=0$, there is nothing to show. In this case $M$ is a $\Zp$-cohomology sphere and therefore also a rational cohomology sphere, and $z|_{\Q}=0$ induces periodicity tautologically. We may assume $z_p\ne 0$.

Then we claim $z$ divides $ax$ for some $a\in \Z$. For this construct a sequence of elements $z_i$, starting with $z_1=z$, by multiplying the previous one with $z$, and then, if the resulting degree is above the degree of $x$, dividing by $x$. This is possible since $x$ induces periodicity. Doing this, eventually one will get a $z_i$ with degree equal to the degree of $x$. Since $\cup x$ is surjective from degree zero, this means $z_i$ is a multiple $ax$ of $x$. To see that it is a nonzero multiple, we may reduce the whole construction mod $p$. Since multiplication by $z|_p$ and division by $x|_p$ are isomorphisms, all $z_j|_p$ are nonzero, so $z_i|_p=(ax)|_p=a|_p x|_p \ne 0$, i.e. $a\ne 0$. Thus $z|_{\Q}$ divides $x|_{\Q}$ and the claim follows from Lemma \ref{ycupz}.
\end{proof}

Now we are in position to state the strategy of four-periodicity, which goes as follows:\\
\begin{enumerate}
\item Given periodicity in $\Hz{M}$, we obtain the same periodicity in $\Hp{M}$ for all $p$.
\item Using tools from mod $p$ cohomology, any given periodicity in $\Hp{M}$ implies periodicity of length $k_p$, depending on $p$, with $k_p$ dividing $2(p-1)p^{a}$.
\item Up to lifting issues, these different periodicities for the primes $p=2$ and $p=3$ can be put together with rational coefficients to obtain 4-periodicity.
\end{enumerate}

We have already seen how (1) and (3) work.

This leaves us with the central question: How do we improve periodicity over $\Zp$?\\
The main tool for this are the so-called reduced $p$-th power maps, due to Steenrod, \cite{steen47}($p=2$), \cite{steen53a, steen53b}($p$ odd), whose key properties we will now give a brief introduction to:\\

Steenrod's reduced $p$-th power maps
$P^i:\Hp{X}\to H^{*+2(p-1)i}(X, \Zp)$ are, for each $p$, a family of maps that exist in the mod $p$ cohomology of every topological space $X$:

$$P^i : H^*(X,\Zp) \rightarrow H^{*+2(p-1)i}(X,\Zp)$$
is a family of additive homomorphisms characterized by naturality and\\

\begin{enumerate}
	\item for $x\in H^*(X, \Zp)$,
	\begin{itemize}
		\item $P^0(x)=x$.
		\item $P^i(x)=x^p$ if $deg(x) = 2i$.
		\item $P^i(x)=0$ if $deg(x) < 2i$.
	\end{itemize}
	\item (Cartan formula) For $x,y\in H^*(X, \Zp)$ we have $P^i(xy)=\sum\limits_{k_1+k_2=i}P^{k_1}(x)P^{k_2}(y)$.
	\item (Adem relations) For $a<pb$, 
	$$P^a \circ P^b = \sum\limits_{j\leq a/p}(-1)^{a+j}\binom{(p-1)(b-j)-1}{a-pj}P^{a+b-j} \circ P^j.$$
\end{enumerate}
\begin{Rmk} For $p=2$, there is conventionally a slightly different notation, where one writes $Sq^{2i}$ for $P^i$ and denote the Bockstein homomorphism $\beta:H^*\to H^{*+1}$ associated to the short exact sequence of coefficients $\Zp \to \mathbb{Z}_{p^2} \to \Zp$ by $Sq^1$ (Then $Sq^{2i+1}=Sq^{1}Sq^{2i}$ to satisfy the Adem relations).\\
One consequence of the Adem relations is that the algebra formed by these maps and $\beta$ has a basis, as a $\Zp$-algebra, of $\{\beta, P^{p^i} | i\ge 0\}$. \\
There are also Adem relations describing interactions between the $P^i$ and $\beta$. However, as $\beta$ will play no role in this work, we omit them.
\end{Rmk}

This, together with the relation $x^p=P^{\frac{|x|}{2}}(x)$, is the starting point for improvement of periodicity in mod $p$ cohomology for Kennard in \cite{k13}.

Before we start, we make the following definition/observation first made in the author's masters' thesis, which will simplify several of the following proofs:

\begin{DefLem}\label{iota} For $y\in H^*(M, \Zp)$, \\
set $\iota(y):=\min\{i\ge 0 | P^i(y) $ induces periodicity$\}$, where $\min\emptyset =+\infty$ by convention.\\

Let $y, z$ be elements of $H^*(M,\Zp)$.
\begin{enumerate}
\item If $\iota(y), \iota(z)<\infty$ and $|P^{\iota(y)}(y)|+|P^{\iota(z)}(z)|\le n$, or
\item if $\iota(yz)<\infty$ and $|P^{\iota(yz)}(yz)|\le n$, 
\end{enumerate}
then $\iota(yz)=\iota(y)+\iota(z)$ and $P^{\iota(yz)}(yz)=P^{\iota(y)}(y)P^{\iota(z)}(z)$.
\end{DefLem}
\begin{proof}
We start with the proof for case (1):

Consider $P^{\iota(y)+\iota(z)}(yz)$. By the Cartan formula we have
\begin{align*}
P^{\iota(y)+\iota(z)}(yz)=P^{\iota(y)}(y)P^{\iota(z)}(z)+\sum\limits_{j+\bar{j}=\iota(y)+\iota(z), (j, \bar{j})\ne (\iota(y), \iota(z))}P^j(y)P^{\bar{j}}(z).
\end{align*}

By assumption, the term outside the sum induces periodicity. For any term in the sum, either $j<\iota(y)$ or $\bar{j}<\iota(z)$. By definition, this means these terms cannot induce periodicity, so must vanish, since cohomology in this degree is singly generated by $P^{\iota(y)}(y)P^{\iota(z)}(z)$.

Thus we have shown that $P^{\iota(y)+\iota(z)}(yz)=P^{\iota(y)}(y)P^{\iota(z)}(z)$, which induces periodicity.

Assume now that some $P^m(yz)$ induces periodicity for $m<\iota(y)+\iota(z)$. Again, by the Cartan formula, we get

\begin{align*}
P^{m}(yz)=\sum\limits_{j+\bar{j}=m}P^j(y)P^{\bar{j}}(z).
\end{align*}

However, since we assumed $m<\iota(y)+\iota(z)$, now every term in the sum has $j<\iota(y)$ or $\bar{j}<\iota(z)$. But now one of these terms must be nonzero and therefore induce periodicity, contradicting the minimality of either $\iota(y)$ or $\iota(z)$. We obtain $\iota(yz)=\iota(y)+\iota(z)$ and are done.\\

If we are instead in case (2), we consider

\begin{align*}
P^{\iota(yz)}(yz)=\sum\limits_{j+\bar{j}=\iota(yz)}P^j(y)P^{\bar{j}}(z).
\end{align*}

Since the left side induces periodicity, there must be at least one nonzero term in the sum, which gives us some pair $j+\bar{j}= \iota(yz)$ such that $P^j(y)$ and $P^{\bar{j}}(z)$ induce periodicity. In particular, we obtain $\iota(y)+\iota(z)\le j+\bar{j} = \iota(yz)$, and we get the claim by observing that now $y,z$ satisfy the assumptions of case (1).

\end{proof}

We prove now the following Lemma, which Kennard proves for the special case where $k$ is minimal. We give here a Lemma that will have suitable generality for use in the later sections:\\

\begin{Lemma}\label{oldgoingdown}
Let $H^*(M, \Zp)$ be $k$-periodic, with $pk\le n$, and let $y\in H^*(M,\Zp)$ be an element such that $P^i(y)=:x$ induces periodicity, with $|x|-|y|\le k$. Then $y$ induces periodicity.
\end{Lemma}

\begin{proof}
First, pick some element $\hat{x}$ inducing periodicity in minimal degree $\hat{k}$. By assumption, $\hat{k}$ divides both $k$ and $|x|$. By potentially replacing $x$, we may assume $i$ is minimal such that $P^i(y)$ induces periodicity, and may assume $i\ne 0$ for the sake of contradiction. Write now $y=y_1 \hat{x}^a$ for some $y_1$ with degree less than $\hat{k}$.

By Lemma \ref{iota}, $\iota(y_1)=\iota(y)=i$, since $\iota(\hat{x})=0$, as $\hat{x}=P^0(\hat{x})$ induces periodicity already. Set $P^i(y_1)=:x_1$.

Assume for the moment $\hat{k}\ne k$. Then $\hat{k}\le \frac{k}{2}$ and $2p\hat{k}\le n$. In particular, $P^*(z)$ can, for $|z|\le \hat{k}$, only take nonzero values in degrees $\le p\hat{k}\le \frac{n}{2}$ and we may multiply two such expressions without issues. Using this we construct a sequence of elements $y_j$ with $|y_j|< \hat{k}, \iota(y_j)=j i$ by the following procedure: 

Consider $y_j y_1$. By Lemma $\ref{iota}$, $\iota(y_j y_1)=\iota(y_j)+\iota(y_1)=ji+i=(j+1)i$. Now define $y_{j+1}$ either by $y_j y_1$, if this is in degree less than $\hat{k}$, or by $\hat{x}y_{j+1}=y_j y_1$. If we are in the second case, $\iota(y_{j+1})=\iota(y_j y_1)$ since $\iota(\hat{x})=0$. 

Of course, eventually one will reach some $y_j$ where this is contradictory, e.g. if $\hat{k}$ divides $j|y_0|$, we see $|y_j|=0$, but then $P^{i_j}(y_j)$ has to vanish for degree reasons and can not induce periodicity.\\
This leaves us with the case $k=\hat{k}$, which is considerably simpler. In this case after scaling one obtains $P^i(y_1)=\hat{x}$, since the condition $|x|-|y|\le k$ now implies $|P^i(y_1)|\le |y_1|+k < 2k$ and by minimality of $\hat{x}$ elements inducing periodicity have degree a multiple of $\hat{k}$. But now one may use the same construction as above, noting that now $P^{ji}(y_j)$ can only increase in degree by at most $|P^i(y_1)|=\hat{k}$ compared to $P^{(j-1)i}(y_{j-1})$ and that $|P^{ji}(y_j)|=|P^{(j-1)i}(y_{j-1})|$ if $|y_{j-1} y_0|\ge \hat{x}$. This means the first time this expression reaches $|P^{ji}(y_j)|=|\hat{x}^p|$, we must have $|y_j|=|y_1 y_{j-1}|<|\hat{x}|$, but this is a contradiction since then $P^{ji}(y_j)=0$ for degree reasons.
\end{proof}

Note: one sees from the construction that the same result holds without any assumptions on $|x|-|y|$ if one assumes $2(p-1)k\le n$. In this case one may check that it is possible to use the above proof for the case $\hat{k}\ne k$ for all cases. \\
Of course, a posteriori one can use Kennard's main result to drop the degree requirements altogether and get the much more elegant\\

\begin{Lemma}\label{pgoingdown}
Let $H^*(M, \Zp)$ be $k$-periodic, with $pk\le n$, and let $P^i(y)=x$ induce periodicity. Then $y$ induces periodicity.
\end{Lemma}
\begin{proof}
Let $\hat{k}$ again be the minimal periodicity, induced by $\hat{x}$. By the above remark, if $\hat{k}<k$ the claim follows from the proof of Lemma \ref{oldgoingdown}, since in this case $M$ is $\hat{k}$-periodic mod $p$ with $2p\hat{k}\le pk \le n$. We may assume $k=\hat{k}$. By Kennard's work, \cite{k13}, $k$-periodicity with $pk\le n$ implies $\hat{k}=2\lambda p^a$, for some $a\ge 0$ and some $\lambda$ dividing $p-1$. Let now $(y,P^i(y))$ be a pair such that $y$ does not induce periodicity, $P^i(y)$ does induce periodicity, and $i$ is minimal among all such pairs. We of course want to show such a pair does not exist. As before, a sum of elements can induce periodicity only if at least one of the summands does. Therefore we may use that the algebra of the $P^i$ is generated by those of the form $P^{p^j}$ to find some $P^{i_1}(P^{i_2}(...P^{i_{m}}(y)...))$ inducing periodicity where the $i_*$ are powers of $p$. But now $m=1$ since otherwise either $(y, P^{i_m}(y))$ or $(P^{i_m}(y), P^i(y))$ contradict the minimality of our chosen pair. We write $i=p^j$. As before, we may reduce to $|y|<\hat{k}=2\lambda p^a$ by dividing out $\hat{x}$. But now we are almost done. If $j>a$, $P^i(y)=0$ for degree reasons. If $j<a$, we are in the situation of Lemma \ref{oldgoingdown}. For $j<a-1$ this is clear, for $j=a-1$ one checks $|P^i(y)|<|\hat{x}^2|$, which by minimality of $\hat{x}$ forces $|P^i(y)|=|\hat{x}|$. If $j=a$, we have that $P^{p^j}(y)$ is a nonzero multiple of some $\hat{x}^b$ and therefore has degree $2\lambda b p^a$. But then $y$ has degree $2\lambda b p^a - 2(p-1)p^a = 2(\lambda b -(p-1))p^a$. However, since $\lambda$ divides $p-1$, it also divides $\lambda b-(p-1)$. Therefore $y$ has the same degree as some power of $\hat{x}$, and since it is nonzero it must be a nonzero multiple of this power, so induce periodicity, a contradiction.
\end{proof}

In fact, from Lemma \ref{oldgoingdown} and the fact that one can decompose $P^{\frac{|x|}{2}}$ into basis elements $P^{p^a}$ where $p^a$ divide $\frac{|x|}{2}$, in the equation $x^p=P^{\frac{|x|}{2}}(x)$, one immediately gets $\hat{k}=2\lambda p^a$ for some $a\ge 0$ and $\lambda\le p-1$. That one can arrange further that $\lambda$ divides $p-1$ requires some extra work showing existence of modified decompositions for $P^{\frac{|x|}{2}}$.

For us, Lemmas \ref{oldgoingdown}/\ref{pgoingdown} will be useful in another way, as it will let us handle easily any case where we are able to improve the background periodicity enough to satisfy the requirements.

\subsection{Improvement using characteristic classes}\,\\

In order to find our improvement, we make the following observation instead of using $x^p=P^{\frac{|x|}{2}}(x)$:

If the vector bundle $\Nu=\Nu(N)$ admits a complex structure, then $e(\Nu)$ is the Chern class $c_{\frac{|e|}{2}}(\Nu)$. Characteristic classes like this are pullbacks of corresponding universal classes in the cohomology of an appropriate classifying space, and, in fact, we will get around using $x^p=P^{\frac{|x|}{2}}(x)$, an equation that requires us to make assumptions about $x^3$ at $p=3$, by pulling back certain relations that we show to hold already in the cohomology of this classifying space. We prove the following Lemma:

\begin{Lemma}\label{appalemma}
Let $c_k\in H^{2k}(BU, \Zp)$ be the mod $p$ reduction of the universal Chern class in degree $2k$, $k=\lambda p^i, p\nmid \lambda>p$. 

Then $c_k$ can be written as a sum of products of lower degree Chern classes and a nonzero multiple of $P^{p^i}(c_{(\lambda-p+1)p^i})$.

In particular, the same relation holds for the Chern classes of any complex vector bundle, since they are given as pullbacks.
\end{Lemma}

\begin{proof}
This question is contained strictly in the problem of determining the leading coefficient (coefficient of $c_k$) in 
\begin{align*}
P^i(c_{k-(p-1)i}) = \sum\limits_{(a_1, ..., a_p)\text{ partition of }k}a(a_1, ..., a_p)c_{a_1}... c_{a_p}.
\end{align*} For the case $p=2$, a formula for all coefficients is due to Wu \cite{wu} and given by $Sq^i(w_{k-i})=\sum_{t=0}^i \begin{pmatrix}k-i-1-t \\ i-t  \end{pmatrix}w_t w_{k-t}$. For general odd primes, the leading coefficient seems to be due to Brown and Peterson \cite{brownpeterson}, who show it to be $\begin{pmatrix}k-(p-1)i - 1 \\ i\end{pmatrix}$ (cf the Wu formula at $t=0$). 

However, as we only became aware of this work after already producing an elementary proof, we include that in Appendix A.
\end{proof}

Having this, we make an improvement to Lemma \ref{oldgoingdown} tailored to the situation in the conclusion of Lemma \ref{appalemma}, this does not use characteristic classes, but characteristic classes will be what ends up producing situations where the assumptions are satisfied.

\begin{Lemma}\label{goingdown}
Let $p$ be an odd prime. Assume $H^*(M, \Zp)$ is $k$-periodic, with $\frac{p+1}{2}k\le n$, and let $y\in H^*(M,\Zp)$ be an element such that $P^{p^i}(y)=:x$ induces periodicity with $|x|-|y| \le k$, where $p^i| |x|$.

Then $y$ induces periodicity.
\end{Lemma}

\begin{proof}

Let $\hat{k}$ be the minimal degree among elements inducing periodicity, and $\hat{x}$ a corresponding element. Now $\hat{k}$ divides $k$ by Lemma \ref{ycupz} / Cor \ref{gcd}. If they are not equal, the claim follows from Lemma \ref{pgoingdown}. We may assume $\hat{k}=k$. Dividing out $\hat{x}$, we may assume $|y|<k$. Note that if $p^i$ divides $|x|$, either it divides $|\hat{x}|$ or $|x|\ge p|\hat{x}|$. Since in the second case we have $p|\hat{x}|\le n$, the Lemma follows from Lemma \ref{pgoingdown} in this case and we may assume $p^i$ divides $|\hat{x}|$, so dividing by $\hat{x}$ preserves the condition of the Lemma. Since, by assumption, we have $|x|-|y|\le k$, we obtain $|P^{p^i}(y)|\le |y|+k <2k$. In particular, by minimality of $\hat{x}$, $x = P^{p^i}(y)$ is a nonzero multiple of $\hat{x}$.

We may therefore assume that there are no elements inducing periodicity in degrees less than $|x|$.

Write $k=|x|=2\lambda p^{\bar{k}}$, where $\bar{k}$ is minimal such that there is a $\bar{y}$ with $P^{p^{\bar{k}}}(\bar{y})=x$. By the assumption of the Lemma, such a $\bar{k}$ exists. Furthermore, proving the Lemma for such a minimal $\bar{y}$ contradicts minimality of $x$, and the claim follows for $y$ as well by the above. We may assume $i=\bar{k}, y=\bar{y}$.

We may assume $\lambda>p$: If $\lambda<p$, $P^{p^{\bar{k}}}(y)=0$ for degree reasons, and if $\lambda=p, x=P^{p^{\bar{k}}}(y)=y^p$, in contradiction to the minimality of $x$.

Starting with $P^{p^{\bar{k}}}(y)=x^1$, we construct equations $P^{jp^{\bar{k}}}(y_j)=x^{a_j}$ with the following properties:

\begin{itemize}
\item $|y_j|<|x|$.
\item $P^{jp^{\bar{k}}}(y_j)=x^{a_j}$, but no $P^*(y_j)$ induces periodicity for $*<jp^{\bar{k}}$, i.e. $\iota(y_j)=ip^{\bar{k}}$.
\item $a_j \le a_{j+1} \le a_j +1$.
\end{itemize}

Our goal with this is to achieve contradiction, which occurs if $p |y_j| < |x^{a_j}|$, since then $P^{jp^{\bar{k}}}(y_j)$ vanishes for degree reasons, but is nonzero by construction.

We assume in the following $a_j \le \frac{p-1}{2}$. We will see later that contradiction occurs before $a_i$ gets larger than this. 

For $j=1$, we have $y_1=y, \iota(y_1)=p^{\bar{k}}$, $P^{\iota(y)}(y)=x$, so $a_1=1$.

For the iteration, consider once more $y_j y_1$. $\iota(y_j y_1)=\iota(y_j)+\iota(y_1)=(j+1)p^{\bar{k}}$ by Lemma $\ref{iota}$, and $P^{(j+1)p^{\bar{k}}}(y_j y_1)=x^{a_j}x^1$.

If $|y_j y_1|<|x|$, we are done and set $y_{j+1}=y_j y_1$, $a_{j+1}=a_j +1$.\\

Otherwise, define $y_{j+1}$ by $y_{j+1}x=y_j y_1$. Again, by Lemma \ref{iota}, we have $\iota(y_{j+1})=\iota(y_j y_1)$ and 
\begin{align*}
x^{a_j +1}=P^{\iota(y_j y_1)}(y_j y_1)=P^{\iota(y_{j+1})}(y_{j+1}x)=P^{\iota(y_{j+1})}(y_{j+1})x,
\end{align*}
so we set $a_{j+1}=a_j$.

Since the above construction uses equations in degree $|x^{a_j +1}|$, it remains to show that we achieve our contradiction while $a_j\le \frac{p-1}{2}$.\\

To see this, we look at what happens to the degree of $y_j$ along the sequence. Following along, we get $|y_j|=|y_{j-1}|+|y|-|x|$ if this is nonnegative, and $|y_j|=|y_{j-1}|+|y|$ otherwise, and that $a_j$ is $a_{j-1}$ in the first and $a_{j-1}+1$ in the second case. This recursion gives the explicit form $|y_j|=2p^{\bar{k}}(a_j\lambda - j(p-1))$, $a_j=\lceil \frac{j(p-1)}{\lambda}\rceil$.

Using this we see that for given $a$, at the last $j$ with $a_j=a$, $y_j$ has degree $2p^{\bar{k}}[a_j \lambda]_{p-1}$, where we take $[\cdot]_{p-1}$ to mean the smallest nonnegative representative of a given class mod $p-1$.

We achieve the desired contradiction at this $y_j$ if 
\begin{align*}
|y_j^p| &< |P^{\iota(y_j)}(y_j)|\\
\iff \hfill p|y_j| &< |x^{a_j}|\\
\iff \hfill p(2p^{\bar{k}}[a_j \lambda]_{p-1})&< a_j |x|= a_j(2p^{\bar{k}} \lambda)\\
\iff \hfill [a_j \lambda]_{p-1} &< \frac{a_j \lambda}{p}.\\
\end{align*}

This inequality is, however, always achieved for some $a_j \le \frac{p-1}{2}$, since 
\begin{align*}
[\frac{p-1}{2} \lambda]_{p-1} = \begin{cases}
0,& \lambda \text{ even}\\
\frac{p-1}{2},& \lambda \text{ odd}
\end{cases} \le \frac{p-1}{2} < \frac{p-1}{2}\frac{\lambda}{p}.\\
\end{align*}
\end{proof}

Note at this point that $\frac{p-1}{2}$ is an optimal bound for the last step, achieved by $\lambda=p+1$.\\

Putting these two lemmas together, we obtain

\begin{Th}[Complex $\Zp$-periodicity Theorem]\label{zpperiod} $\,$\\
Let $M$ be a closed orientable manifold, $p$ an odd prime.
\begin{itemize} 
\item Let, for some complex vector bundle $V$ on $M$, $c_k(V)$ induce periodicity in $\Hp{M}$ with $2k\frac{p+1}{2}\le n$.\\
Then $c_i(V)$ induces periodicity in $\Hp{M}$ for some $i=\lambda p^a\le k$, where $p^a|k, \lambda\le p-1$.
\item Let, for some complex vector bundle $V$ on an orientable submanifold $N$ with maximally connected inclusion and $dim(V)\le \dim(N)$, $(\iota^*)^{-1}c_k(V)$ induce periodicity in $\Hp{M}$ with $2k\frac{p+1}{2}\le n$.\\
Then $(\iota^*)^{-1}c_i(V)$ induces periodicity in $\Hp{M}$ for some $i=\lambda p^a\le k$, where $p^a|k, \lambda\le p-1$.
\end{itemize}
\end{Th}
\begin{proof}
If we are in the second case, $\iota^*$ is an isomorphism mod $p$ in degrees $\le |c_k| \le \dim(N)$ unless $M$ is a mod $p$ cohomology sphere, in which case the Theorem holds trivially, so we may assume $\iota^*$ to be an isomorphism. Then, by naturality of the Steenrod maps and cup product, the $(\iota^*)^{-1}c_i$ satisfy all the same relations as the $c_i$ in degrees $\le |c_k|$. For this reason, we suppress the notation and write $c_i$ for $(\iota^*)^{-1}c_i$ in this case as well, understanding that on $M$ they may not truly be Chern classes, but satisfy the same relations. Using this slight abuse of notation, both cases have the same proof:\\

Lemmas \ref{appalemma} and \ref{goingdown} imply everything except that $p^a|k$, as the class $c_{(\lambda-(p-1))p^i}$ produced by Lemma \ref{appalemma} may have degree divisible by a higher power than $p^i$. We show that in this situation, one of the product terms in the conclusion of Lemma \ref{appalemma} must be nonvanishing as well. This will imply the claim of the theorem as some factor in the product must have degree with the same or lower divisibility by $p$ as the total degree of the product.

Let us assume for the sake of contradiction that we are in a situation where $c_{\lambda p^a}$ and $P^{p^i}(c_{\lambda p^a})=z c_{\lambda p^a + (p-1)p^i}$ induce periodicity in $M$ with $i<a$, $z\in \Zp^{*}$, $p\nmid \lambda$.

Let $\hat{x}$ be an element inducing periodicity in $M$ with minimal degree, and write $\hat{k}:=|\hat{x}|=2\mu p^b$, where $p\nmid \mu$. By minimality, $\hat{k}$ divides $|c_{\lambda p^a}|=2\lambda p^a$. After replacing $\hat{x}$ with a scalar multiple, we may assume $\hat{x}^{\frac{\lambda}{\mu}p^{a-b}}=c_{\lambda p^a}$. Note that since $\hat{k}$ also divides $|c_{\lambda p^a + (p-1)p^i}|$, we have $b\le i < a$. But now 
\begin{align*}
P^{p^i}(c_{\lambda p^a})=P^{p^i}(\hat{x}^{\frac{\lambda}{\mu}p^{a-b}})=\begin{cases}\big(P^{p^{i-(a-b)}}(\hat{x}^{\frac{\lambda}{\mu}})\big)^{p^{a-b}}&\text{, if } i\ge a-b\\0 &\text{otherwise}.\end{cases}
\end{align*}
This element induces periodicity by assumption, so we see that $P^{p^{i-(a-b)}}(\hat{x}^{\frac{\lambda}{\mu}})$ induces periodicity as well. Of course, as a power of $\hat{x}$, the same is true for $\hat{x}^{\frac{\lambda}{\mu}}$. But now we have two elements inducing periodicity whose degrees differ by $2(p-1)p^{i-(a-b)}=2(p-1)p^{b-(a-i)}$. Now $\hat{k}$ must divide this number, but this is impossible since $\hat{k}$ is a multiple of $p^b$.

From this, we see that $P^{p^i}(c_{\lambda p^a})$ must be zero in $M$ in such a situation, but that means the product terms cancel the nonzero $c_{\lambda p^a + (p-1)p^i}$-term. Therefore at least one product term is nonzero in $M$ and we are done.
\end{proof}

\begin{Rmk}[on $\lambda$] If one is familiar with the earlier work on periodicity, one may have expected that it is also possible to arrange that $\lambda$ divides $p-1$. This would be nice, first aesthetically, but also, similarly to the difference between Lemmas \ref{goingdown} and \ref{oldgoingdown}, would allow one to drop the requirement on $|x|-|y|$ in Lemma \ref{pgoingdown} when periodicity is induced by a chern class. However, using just the results at hand, $\lambda | p-1$ does not follow. To see this, imagine for a moment the starting periodicity is given simply by $c_\lambda$ for some $\lambda<p$ not dividing $p-1$. There is no hope in finding this class either in products or $P^i$ coming from lower degrees. In fact, since $P^1$ already increases degree by $2(p-1)$, there are no such maps at all with image in degree less than $2p=|c_p|$.

However, one may still have some hope for this stronger conclusion. For this, let us consider the smallest example where it may fail, $p=5$ with $c_3=e(\Nu(N))$ inducing periodicity: 

Now, instead of looking at $c_3$, one can look at $c_3^2$ instead and compute $P^1(c_2)=2c_3^2+c_6+R$, where terms in $R$ all have either $c_1$ or $c_2$ as factor. However, since $3|c_3|=3\codim(N)\le n$, we have $2|c_3|\le \dim(N)$, so $c_6$ still has degree low enough to pull back to $N$. But of course we assumed the bundle to have Euler class $c_3$, so $c_6$ is zero on $N$, therefore on $M$.

If now also no term in $R$ induces periodicity, we have $P^1(c_2)=2c_3^2$, so $c_2\ne 0$ and one considers $P^1(c_2^2)=2P^1(c_2)c_2=4c_3^2 c_2$. Since $c_3$ induces periodicity, the right side is nonzero in $M$. Dividing $c_2^2$ by $c_3$ one sees $P^1(y_1)\ne 0$ for the quotient $y_1$. But then for degree reasons we see that $P^1(y_1)=y_1^5$ is nonzero, so $y_1^3$ is nonzero as well and therefore a multiple of $c_3$ in $M$, and one gets that $c_2$ induces periodicity. Note here that in order to keep our calculations below degree $\frac{5+1}{2}|c_3|=3|c_3|$, we had to give up on showing that some $P^*(c_2^2)$ induces periodicity and instead only found some $P^*(c_2^2)$ that was nonzero.

It is a priori unclear whether techniques like these can show the stronger conclusion in general. Any attempt along these lines seems to require at least the knowledge of some additional coefficients in $P^*(c_{ik-(p-1)*})$ for some suitable set of partitions, which is not known to us to be readily available, and to use the vanishing in $N$ of some high degree classes, which may make induction hard. Note that it is possible to write down configurations like above with nonzero $c_6$ that are consistent up to degree $3|c_3|$ where all $c_i, 3\nmid i$ vanish, i.e. $c_6=-2c_3^2, c_9=c_3^3$, so using vanishing above the dimension of the vector bundle was essential.
\end{Rmk}

The corresponding statement for $p=2$ is the following:\\

\begin{Th}[Complex $\Ztwo$-periodicity Theorem]
\,\\
Let $M$ be a closed manifold.
\begin{itemize}
\item Let, for some complex vector bundle $V$ on $M$, $c_k(V)$ induce periodicity in $\Htwo{M}$ with $4k=2|c_k|\le n$.

Then $c_i(V)$ induces periodicity in $\Htwo{M}$ for some $i=2^a | k$.

\item Let, for some complex vector bundle $V$ on a submanifold $N$ with maximally connected inclusion and $dim(V)\le \dim(N)$, $(\iota^*)^{-1}c_k(V)$ induce periodicity in $\Htwo{M}$ with $4k=2|c_k|\le n$.

Then $(\iota^*)^{-1}c_i(V)$ induces periodicity in $\Htwo{M}$ for some $i=2^a | k$.
\end{itemize}
\end{Th}
\begin{proof}
Except for the condition $i|k$ this follows already from Lemma \ref{oldgoingdown} and the well known Wu-formula (note that for complex vector bundles, the Stiefel-Whitney class $w_{2k}$  is the mod 2 restriction of $c_k$, and all Stiefel-Whitney classes $w_i$ for odd $i$ vanish).\\

To see one can arrange also $i|k$ one can use the same idea as in the proof of Theorem \ref{zpperiod}.
\end{proof}

\begin{Rem}
Using the Wu formula, one may prove the preceding theorem using wu classes of a real vector bundle instead of chern classes of a complex vector bundle, and arrive at the same conclusion. However, as we will need later that the elements induced by our $\Zp$-periodicity theorems naturally lift to integral cohomology, we choose to prove it for mod 2 chern classes instead.
\end{Rem}

Using the tools now at our disposal, we are ready to prove the main theorem:\\

\begin{proof}[Proof of Theorem \ref{mainper}]\,\\

We begin with the complex case:

If $\Nu(N)$ carries a complex structure, $e(\Nu)=c_\frac{k}{2}$ is a Chern class, inducing periodicity with $4\frac{k}{2}\le n$ by assumption.

Therefore, the $\Ztwo$ and $\Z_3$-periodicity Theorems are satisfied and we find two Chern classes $c_{2^a}$ and $c_{\lambda 3^b}$ whose reductions induce periodicity in mod 2 and mod 3 cohomology, respectively, for some $a,b\ge 0, \lambda\le 2$.

As we have already seen, the images of these Chern classes in rational cohomology induce periodicity by Lemma \ref{ptoq}, so we obtain in total $\gcd(2^{a+1}, 2\lambda 3^b, k)$-periodicity, which is either $2$- or $4$-periodicity.\\

In the noncomplex case, we use the relation $e(\Nu)^2=c_k(\Nu \tensor_{\R}\C)$ in $\Hz{N}$ that holds for any real vector bundle.

From here we continue as in the complex case, noting that we have $e(\Nu)$ inducing background periodicity such that we may use Lemma \ref{pgoingdown} instead of Lemma \ref{goingdown}. Note that Lemma \ref{goingdown} does not apply a priori if $|x|=2k$.
\end{proof}

\newpage

\section{A Conjecture of Hopf with Symmetry}\label{B}

In this section, we will use the results of Section \ref{A} to prove the following theorem, which we will then use to obtain the Hopf conjecture with $T^4$-symmetry as corollary:

\begin{customtheorem}{\ref{mainfix}}
Let $(M^n, g)$ be a closed, orientable Riemannian manifold of positive sectional curvature equipped with an effective isometric $T^4$-action, and $F^f$ be a component of the fixed point set $Fix(T^4)$. 

Then $H^*(F, \Q)$ is, as a ring, that of $\s{f}, \CP{f}$, or $\HP{f}$.
\end{customtheorem}

The main tool for relating positive curvature to topological restrictions is the famous Connectedness Lemma of Wilking, 

\begin{theorem*}[Connectedness Lemma]\cite{w03}
Let $M^n$ be a compact Riemannian manifold with positive sectional curvature.
\begin{enumerate}
\item Suppose that $N^{n-k}\subset M^n$ is a compact totally geodesic embedded submanifold of codimension $k$. Then the inclusion map is $(n-2k+1)$-connected.
\item Suppose that $N_1^{n-k_1}$, $N_2^{n-k_2}\subset M^n$ are two compact totally geodesic embedded submanifolds, $k_1\le k_2, k_1+k_2\le n$. Then the intersection $N_1\cap N_2$ is a totally geodesic embedded submanifold as well, and the inclusion $N_1^{n-k_1}\cap N_2^{n-k_2} \into N_2^{n-k_2}$ is $(n-k_1-k_2)$-connected.
\end{enumerate}
\end{theorem*}
Note here that, if in addition the intersection of $N_1^{n-k_1}$ and $N_2^{n-k_2}$ is transversal, $n-k_1-k_2$ is the dimension of $N_1\cap N_2$, and we are in the situation of periodicity.

For the proof of Theorem \ref{mainfix}, we use the methods of \cite{kww}. Our aim is to reduce to the following, which follows from the aforementioned work and Theorem \ref{mainper}:

\begin{Lemma}\label{stdorsxcp}
Let $T^2 \acton N$ be an effective isometric action splitting at $p\in F^f$ as a product of two circle actions with weights $e_{1,2}$, with $sec(N)>0$, such that the inclusion $\fk{e_1}\subset N$ is maximally connected with $\dim \fk e_1 \ge \frac{\dim N}{2}$.

Then $H^*(F, \Q)$ is, as a ring, that of $\s{f}, \CP{f}$, or $\HP{f}$, or there is a fixed point component $F'$ of $T^2$ of rational cohomology type $S^h\times \cp^{k\ge 1}$ for $h\in\{2,3\}$.
\end{Lemma}
\begin{proof}Denote by $\mu(e_i)$ the multiplicity of $e_i$ as weights of the isotropy action at $p$.

We say that $F$ is standard if it has the rational cohomology ring of a sphere or complex or quaternionic projective space.

If $f\le 3$, $F$ is automatically standard.\\
If $f=4, \mu(e_1)=1$, the inclusion $F\into \fk{e_1}$ is $6-2\cdot 2+1=3$-connected by the Connectedness Lemma. Since $\fk{e_1}$ and $N$ are 4-periodic by Theorem \ref{mainper}, $b_2(F)=b_2(\fk e_1)\le 1$ and $F$ is standard by Poincar\'e-duality.\\

In the remaining case, $f+2\mu(e_1)\ge 7$.

By \cite[Thm 3.1]{kww} and Poincar\'{e}-duality, this implies $N$ is standard or $S^h \times \hp^k$ for $h\in\{2,3\}$.

If $N$ is standard, so is $F$. If $N$ is $S^h \times \hp^k$, $F$ can be standard, $S^h \times \hp$ or $S^h \times \cp$. In the last case, our claim is satisfied with $F'=F$. If now $F$ is of type $S^h \times \hp$, the same follows for $\fk e_2$: $N$ has this type, so $\fk e_2$ can also only be standard, $S^h \times \hp$ or $S^h \times \cp$. The only one of these that can have $S^h \times \hp$ as fixed point component is $S^h \times \hp$.

This places us in the situation of \cite[Thm 4.1]{kww}. As in \cite[Proof of Thm A]{kww}, option (b) contradicts the classification of fixed point homogeneous positively curved manifolds by Grove and Searle \cite{grovesearle}, and can never occur. We are therefore in option (a) and find another fixed point component of type $S^h \times \cp$ as claimed. Note that the case $S^1\times \cp$ can never occur as positively curved manifolds have $b_1(M, \Q)=0$.
\end{proof}

\begin{Lemma}\label{torusreduction}
Let $T^4 \acton M$ be an effective torus action on a closed orientable manifold of positive sectional curvature, $F$ a fixed point component of this action and $p\in F$. \\

Then there is a codimension 1 subgroup $H\subset T$ such that the action $T^3=T^4/H \acton \Fix(H)$ satisfies
\begin{itemize}
\item $\Fix(T^3\acton \Fix(H)) = \Fix(T^4 \acton M)$.
\item $\Fix(H)$ is orientable.
\item $T^3$ has no finite isotropy groups near $p$.
\item the isotropy representation at $p$ splits as $S^1\times T^2$ such that the weight corresponding to the $S^1$-factor has minimal multiplicity among all occuring weights.
\end{itemize}
\end{Lemma}
\begin{proof}
We construct the claimed action in two steps. First, if $T^4\acton M$ has a finite isotropy group near $p$, we choose a maximal finite isotropy group $I$ near $p$ and consider the action $\tilde{T}^4:=T^4/I \acton \Fix(I)$. Since $I$ was chosen maximally, this action will have no finite isotropy groups near $p$. The orientability of $\Fix(I)$ was proven in \cite[Thm 1.8]{kww} and is the only point in this Lemma where positive curvature enters. 

It turns out that having no finite isotropy groups near $p$ places a strong restriction on the isotropy representation $\tilde{T}^4 \acton \Nu(F\subset \Fix(I))$. 

First, since the action is assumed to be effective, the set $\Omega \subset \Z^4$ of weights of non-trivial irreducible subrepresentations at $x$ contains a basis of $\Q^4$.

Then the condition on the nonexistence of finite isotropy groups becomes this: For any lattice spanned by elements of $\Omega$, this lattice can not have index greater than $1$ in $\Z^4$. It turns out that there are, up to change of basis, only 17 possible weight sets satisfying these conditions, which are listed in Table \ref{tab:t4list}. We procceed to find our $T^3$ in three cases:

First, if there is a $S^1\subset\tilde{T}^4$ such that the action of $\tilde{T}^4/S^1$ splits at $p$ as $S^1\times S^1 \times S^1$, we are done. In this case we let the $S^1$ whose representation has minimal multiplicity be the $S^1$ of the claim and group the others into a $T^2$. We say the $\tilde{T}^4$-representation is of `splitting $T^3$' type.

By consultation of Table \ref{tab:t4list}, one now sees that of the 17 weight sets, 15 are of splitting $T^3$ type, and are left with the other two possibilities:

(1) The weight set is given by $Std_2 \oplus Std_2 := \{e_1, e_2, e_1-e_2\} \oplus \{e_3, e_4, e_3-e_4\}= \{e_i, e_1-e_2, e_3-e_4 | i\le 4\}$:

This case is not difficult. After change of basis we may assume $e_1$ to have minimal multiplicity (note that our representations are real representations, so we are working up to sign). Then for $S^1=\ker(\langle e_1, e_3, e_4\rangle)$ the weight set of $T^4/S^1$ at $p$ consists of the old weight set intersected with $\langle e_1, e_3, e_4\rangle$, which is $\{e_1, e_3, e_4, e_3-e_4\}= \{e_1\}\oplus \{e_3, e_4, e_3-e_4\}$, splitting as claimed, and the multiplicity of $e_1$ is minimal by construction.\\

(2) The weight set is given by $Std_4 := \{e_i, e_j-e_k | i\le 4,  j < k\le 4\}$:

After change of basis, we may assume the weight with minimal multiplicity is $e_1$ or $e_1-e_2$. Now we note $Std_4\cap \langle \{e_1, e_i-e_j| 2\le i<j\}\rangle = \{e_1\}\oplus \{e_i-e_j|2\le i<j\}$ as well as $Std_4\cap \langle \{e_1-e_2, e_i| i\ge 3\}\rangle = \{e_1-e_2\}\oplus \{e_3, e_4, e_3-e_4\}$. 
Thus we see, in either case, that we obtain a quotient torus $T^3=\tilde{T}^4/S^1 = T^4/(S^1\cdot I)$ with the required properties. Note here that orientability is automatic when moving to $S^1$-fixed point components, and the condition $\Fix(T^3\acton \Fix(H)) = \Fix(T^4/H \acton \Fix(H)) = \Fix(T^4 \acton M)$ is trivial.
\end{proof}

\begin{table}[!h]
  \begin{center}
    \caption{List of connected isotropy $T^4$ weight sets}
    \label{tab:t4list}
    \begin{tabular}{r|c|l} 
      \textbf{Size} & \textbf{weights not including standard basis} & \textbf{Splitting $T^3$?}\\
      \hline
      4 & $\emptyset$ & \checkmark\\
      5 & $\begin{pmatrix}1\\ 0\\ 0\\ -1\end{pmatrix},\begin{pmatrix}1\\ 0\\ -1\\ -1\end{pmatrix},\begin{pmatrix}1\\ -1\\ -1\\ -1\end{pmatrix}$ & \checkmark\\
      6 & $\begin{pmatrix}1 & 1\\ 0 & 0\\ 0 & -1 \\ -1 & 0\end{pmatrix},\begin{pmatrix}1 & 1\\ 0 & -1\\ 0 & -1 \\ -1 & 0\end{pmatrix},\begin{pmatrix}1 & 1\\ 0 & -1\\ -1 & 0 \\ -1 & -1\end{pmatrix},\begin{pmatrix}1 & 0\\ -1 & 0\\ 0 & 1 \\ 0 & -1\end{pmatrix}$ & all except last\\
      7 & $\begin{pmatrix}1 & 1 & 1\\ 0 & 0 & 0\\ 0 & -1 & -1 \\ -1 & 0 & -1\end{pmatrix}, \begin{pmatrix}1 & 1 & 1\\ 0 & 0 & -1\\ 0 & -1 & 0 \\ -1 & 0 & 0\end{pmatrix}, \begin{pmatrix}1 & 1 & 1\\ 0 & 0 & -1\\ 0 & -1 & 0 \\ -1 & 0 & -1\end{pmatrix}, \begin{pmatrix}1 & 1 & 1\\ 0 & 0 & -1\\ 0 & -1 & -1 \\ -1 & 0 & -1\end{pmatrix}$ & \checkmark\\
      8 & $\begin{pmatrix}1 & 1 & 1 & 1\\ 0 & 0 & 0 & -1\\ 0 & -1 & -1 & 0 \\ -1 & 0 & -1 & 0\end{pmatrix} , \begin{pmatrix}1 & 1 & 1 & 1\\ 0 & 0 & -1 & -1\\ 0 & -1 & 0 & -1 \\ -1 & 0 & -1 & 0\end{pmatrix}$ & \checkmark\\
      9 & $\begin{pmatrix} 1 & 1 & 1 & 1 & 1\\ 0 & 0 & 0 & -1 & -1\\ 0 & -1 & -1 & 0 & 0 \\ -1 & 0 & -1 & 0 & -1\end{pmatrix}, \begin{pmatrix} 1 & 1 & 1 & 1 & 1\\ 0 & 0 & -1 & -1 & -1\\ 0 & -1 & 0 & -1 & -1 \\ -1 & 0 & -1 & 0 & -1\end{pmatrix}$ & \checkmark\\
      10 & $\begin{pmatrix} 0 & 0 & 0 & 1 & 1 & 1\\ 0 & 1 & 1 & 0 & 0 & -1\\ 1 & 0 & -1 & 0 & -1 & 0 \\ -1 & -1 & 0 & -1 & 0 & 0\end{pmatrix}$ & $\times$
    \end{tabular}
  \end{center}
\end{table}
\newpage

\begin{proof}[Proof of Theorem \ref{mainfix}]\,\\
We start by applying Lemma \ref{torusreduction} to our $T^4$-action in order to obtain a $T^3$-action that splits near $p\in F^f$ as $S^1\times T^2$ in a way that the weight $e_1$ corresponding to the $S^1$-factor has minimal multiplicity among all weights. Let $\Omega$ be the weight set of $T^3$ at $p$. 

Let $A, B\subset \Omega$ be subsets. For the two submanifolds $\fk(\langle A \rangle)$ and $\fk(\langle B \rangle)$ to intersect transversally, one can check on the level of weights that $\Omega \cap (\langle A \rangle \cup \langle B \rangle) = \Omega \cap \langle A \cup B \rangle$ and $\Omega \cap (\langle A \rangle \cap \langle B \rangle) = \Omega \cap \langle A \cap B \rangle$, as this is equivalent to the condition that the codimension of $\fk(\langle A \cap B \rangle)$ in $\fk(\langle A \cup B \rangle)$ is the sum of the codimensions of $\fk(\langle A \rangle)$ and $\fk(\langle B \rangle)$.\\

Then, to see whether the action satisfies Lemma \ref{torusreduction}, one checks two things:

First, one checks that the remaining action of $T^3/\ker(\langle A \rangle\cup \langle B\rangle)$ on $\fk(\langle A \rangle\cup \langle B\rangle)$ is effective. This is automatic since the action of $T^3$ on $M$ was  effective to start with.

Then one has to check the condition on the dimensions, which can be determined from
\begin{align*}
\dim\fk(\langle A \rangle\cap \langle B\rangle) = f + \sum\limits_{\omega \in \langle A \rangle\cap \langle B\rangle \cap \Omega} 2\mu(\omega)\\
\dim\fk(\langle A \rangle) = f + \sum\limits_{\omega \in \langle A \rangle \cap \Omega} 2\mu(\omega)\\
\dim\fk(\langle B\rangle) = f + \sum\limits_{\omega \in \langle B\rangle \cap \Omega} 2\mu(\omega)
\end{align*}
 
To return to the proof, pick now some weight $e_2$ in the wieght set $\Omega_2\subset \Omega$ corresponding to the $T^2$-factor.

We consider $A=\{e_1, e_2\}, B=\Omega_2$. This intersection is transversal, so the inclusion of $\fk(A\cap B) = \fk(e_2)$ into $\fk(\langle e_1, e_2 \rangle)$ is maximally connected by the Connectedness Lemma. Note here that $\dim\fk(\langle e_1, e_2\rangle)\le \dim\fk \Omega_2$ since $e_1$ has minimal multiplicity and $\Omega_2$ contains at least one weight different from $e_2$.

Then we obtain $\dim \fk(\langle e_1, e_2\rangle) = f+ 2\mu(e_1)+2\mu(e_2) \le 2(f+\mu(e_2))=2\dim \fk(e_2)$. 
  
 Thus we have seen that we can find a codimension two subgroup such that the remaining $T^2$-action on the fixed point component containing $F$ satisfies the assumptions of Lemma \ref{stdorsxcp}. Thus either we are done or we find another fixed point component $F'$ of type $S^h \times \cp^{l\ge 1}$, $h\in\{2,3\}$. Like above, we find a codimension 2 subgroup such that the remaining $T^2$ action on the fixed point component $N'$ over $F'$ is as above. Since $N'$ is $4$-periodic by Theorem \ref{mainper}, as in Lemma \ref{stdorsxcp} we find that $N'$ must be of rational type $S^h \times \hp$. Using that the $T^2$-action we constructed has at $F'$ only two linearly independent weights, the fixed point components $\fk e_1$ and $\fk e_2$ intersect transversally, so $F'\into \fk e_i$ is maximally connected for some $i\in\{1,2\}$ by the Connectedness Lemma. This, however, gives us a contradiction. Since $\fk e_i$ is a fixed point component of a circle action on $N'$, the possible types are standard, $S^h \times \hp$ and $S^h \times \cp$. None of these contains a maximally connected $S^h \times \cp$ of strictly smaller dimension.
\end{proof}

Using this, we get progress on the Hopf conjecture:

\begin{Cor}[Hopf conjecture with $T^4$-symmetry]
Let $(M^n, g)$ be an even-dimensional closed Riemannian manifold with positive sectional curvature and an effective isometric $T^4$-action. Then $\chi(M)>0$.
\end{Cor}
\begin{proof}
We always have $\chi(M)=\chi(Fix(T^4))=\sum\limits_{F \subset \Fix(T^4)}\chi(F)$.

By a result of Berger, $Fix(T^4)$ is nonempty. If $M$ is orientable, by Theorem \ref{mainfix}, each component contributes at least $+1$ to the sum. If $M$ is not orientable, we pass to the orientation cover. By the above argument it has positive Euler characteristic, and by multiplicativity of the Euler characteristic it is twice the Euler characteristic of $M$.
\end{proof}

\section{Application}\label{C}

The goal of this section is to give an application of Theorem \ref{mainfix}, and prove\\

\begin{customtheorem}{\ref{maincomb}}
Let $M^n$ be an even-dimensional closed oriented manifold with $b_{odd}(M)=0$, equipped with a $T^d$-action, $d\ge 5$, such that all fixed point components of codimension-2 subtori are of rational cohomology $\s{}, \cp$, or $\hp$ type.

Then $M$ has the rational cohomology of $\s{n}, \CP{n}$, or $\HP{n}$.
\end{customtheorem}\,\\

Note here that the condition on the odd Betti numbers may be replaced by assuming the action to be equivariantly formal. Both conditions are conjecturally void in the case of positive sectional curvature. By a conjecture of Bott (cf \cite{gh}), positively (or, more generally, non-negatively) curved manifolds are rationally elliptic. Since by Corollary \ref{hopf} the Euler characteristic of $M$ is positive under the conditions of Theorem \ref{maincomb}, the Bott-Grove-Halperin conjecture would already imply vanishing of the odd rational Betti numbers under these conditions by work of Halperin \cite{halp}. \\

We prove Theorem \ref{maincomb} by giving an improvement of \cite[Thm C]{kww}. The main difference in our proof is the use of theorems of real projective combinatorics. \\

\begin{Rem}
 Unlike \cite[Thm C]{kww}, this theorem does not require assumptions on any fixed point components of codimension 3 subtori. In exchange, we can no longer handle the case of $T^4$-actions any more, as the methods we use in the proof of Lemma \ref{atmostfour} do not give rigidity; this also means we no longer encounter the Cayley projective plane $\cayley$ as a cohomology model for $M$. See however Remark \ref{t4remark} for further discussion of the case $d=4$.\\
\end{Rem}
 
 The starting point for proving this theorem is a construction of \cite{kww}, inspired by  the study of GKM actions:\\
 
\begin{Def} \cite[Def 5.1, adapted for notation]{kww}\label{graphdef}\\
\begin{enumerate}
\item If $N^m$ has the rational cohomology of an even-dimensional rank one symmetric
space, we define $\mu(N) \in \{1, 2, 4, m/2\}$ by the property that the cohomology
of $N$ is generated in degree $2\mu(N)$.
\item Suppose an action of $T^d$ on a closed, orientable manifold has the property that any codimension one torus has only fixed-point components with the rational
cohomology of $\s{}$, $\cp$, or $\hp$ type. Consider a connected component $C$ of the one-skeleton $\{p\in M | \dim(Tp)\le 1\}$. We let $F_1, . . . , F_{k_0}$ denote the different components of $F\cap C$.

We define a graph with vertices $\{F_i\}$ as follows: For each codimension one torus $R\subset T^d$, we choose a weight $r\in \mathfrak{t}^{*}$ whose kernel is given by the Lie subalgebra of $R$. We put $\mu(N)$ edges between $F_i$ and $F_j$ with label $r$ if $F_i$ and $F_j$ are contained in the same fixed-point component $N$ of $R$.

If the fixed-point component at $F_i$ has positive dimension, we put $\mu(F_i)$
edges from $F_i$ to itself with label $0$. Moreover, if $F_i$ has positive
dimension and $N \supset F_i$ is the fixed-point component of a codimension one torus $R$ we put $(\mu(N)-\mu(F_i))$ edges with label $r$ from $F_i$ to itself.\end{enumerate}
\end{Def} 
 
\begin{Rem}
From the definition it is clear that the weights that are used as labels here are defined uniquely only up to scaling, i.e. the subtori $R$ define naturally a labelling of the edges with labels in $\mathbb{P}(\mathfrak{t}^*)$.
\end{Rem} 
 
Kennard, Wiemeler and Wilking then prove the following:

\begin{Lemma}[{\cite[Lemmas 5.2-5.5]{kww}}]\label{graphlemma}\,\\
Suppose an action of $T^d$ on a closed manifold $M^n$ has the property that any codimension two torus only has fixed point components of rational type $\s{}$, $\cp$, and $\hp$. Choose $C$ and $F_1, ..., F_{k_0}$ and the corresponding graph, as in Definition \ref{graphdef}.
\begin{enumerate}
\item There exists $m\ge 1$ such that there are precisely $m$ edges between any two vertices, and from $F_i$ to itself whenever $\dim(F_i)\ne 0$. 
\item For every subspace $V$ of $\mathfrak{t}^{*}$, every connected component of the 'reduced' graph containing only the edges with label in $V$ satisfies $(1)$. If $\dim(V)\le 2$, then it satisfies $(1)$ with $m\le2$.
\item $\chi(C)=\sum\chi(F_i)=\frac{n}{2m}+1$.
\item If some fixed point component $F_i$ has $\dim(F_i)\ge 4$, then $m\in\{1,2,\frac{n}{2}\}$.
\item If $b_{odd}(M, \Q)=0$, $d\ge 4$, and $m\in \{1,2,\frac{n}{2}\}$, then $M$ has rational type $\cp, \hp, \s{}$ respectively.
\end{enumerate}
\end{Lemma}
 
 Furthermore, due to a theorem of Chang and Skjelbred \cite{cs73}, the one-skeleton is connected under the assumptions of Theorem \ref{maincomb}.
 
By a combination of $(4)$ and $(5)$, Theorem \ref{maincomb} is clear unless all fixed point components are of dimension zero or two.

 In the following, we use the word `triangle' to mean subgraphs obtained by choosing three vertices and all edges between them. We use interchangeably the word edge for the single edges of the graph as defined above and for the collections of edges between $F_i$ and $F_j$, which we denote $(F_i, F_j)$. By the weight set of some $(F_i, F_j)$ we mean the set $A_{ij}$ of weights of the labels, counted with multiplicity.
 
 For a weight $a\in A_{ij}$, we say the multiplicity $\mu(a)$ is the number of weights in $A_{ij}$ linearly dependent with $a$. \\
  
 We proceed towards proving Theorem \ref{maincomb} by showing a sequence of Lemmas:\\ 
 
As before, we say a manifold is standard if it has rational cohomology type  $\s{}$, $\cp$, or $\hp$

\begin{Lemma}\label{spheregraph}
In the situation of Lemma \ref{graphlemma}, unless $M$ is a rational sphere, we are in one of the following two situations:
\begin{itemize}
\item all fixed points are isolated and there are at least 3 fixed point components
\item there is a fixed point component of type $S^2$ and there are at least two fixed point components
\end{itemize}
\end{Lemma} 

\begin{proof}
There are only two cases to consider that are not allowed by the conclusion:
 
 First, if the graph has two vertices which are isolated fixed points. Then, by (3), $m=\frac{n}{2}$ and $M$ is a sphere.
 
 If there is only one vertex, it must be of type $S^2$, since for any codimension 2 subtorus we have $\chi(F)=\chi(Fix(T^2))\ge 2$ by assumption. Again, by (3), we have $m=\frac{n}{2}$ and $M$ is a sphere.
\end{proof}

\begin{Lemma}\label{s2comb} Assume we are in the situation of Lemma \ref{graphlemma}. Let $F_i$ be a fixed point component of type $S^2$ and $F_j$ another fixed point component. Denote by  $\Omega$ the set of weights between $F_i$ and $F_j$ and call the set of nonzero weights from $F_i$ to itself $N$. Denote by $\tilde{\Omega}, \tilde{N}$ the images in the projectivization $\mathbb{P}(\mathfrak{t}^*)$. Then the following hold:
\begin{itemize}
\item On every line through two points of $\tilde{\Omega}$ lies a point of $\tilde{N}$.
\item On every line through a point of $\tilde{\Omega}$ and a point of $\tilde{N}$ lies another point of $\tilde{\Omega}$.
\end{itemize}
\end{Lemma}
\begin{proof}
We start by showing $\mu(\omega)=1$ for all $\omega\in \Omega$. If this is not the case, by definition of the graph the fixed point component $F^{\omega}_i$ of the codimension one subtorus $\ker \omega$, acting on $M$, containing $F_i$ is of type $\hp$. Since $|N|=m-1\ge \mu(\omega)-1\ge 1$, $F^{\omega}_i$ has an edge leading to itself in the graph of that action. But this is impossible, since the fixed point component $F^{\langle \omega, n \rangle}_i$ of the codimension $2$ torus determined by such an edge with weight $n$ must have the same Euler characteristic as $F^{\omega}_i$ and higher dimension. As $F^{\omega}_i$ was already of type $\hp$, there is no model for $F^{\langle \omega, n \rangle}_i$ with this property.\\

We may now prove the main claim: If we are in situation (1), let $\omega_1, \omega_2$ be distinct elements of $\Omega$ and consider the graph obtained by intersection with $\langle \omega_1, \omega_2 \rangle$, which is the graph of the $T^2$ action on the fixed point set of the codimension two subtorus $\ker(\langle \omega_1, \omega_2 \rangle)$. This reduced graph now has $m_{red}\le 2$ by Lemma \ref{graphlemma} (2), so it has $m_{red}=2$ since we have two weights by construction. We have then $|N \cap \langle \omega_1, \omega_2 \rangle|=m_{red}-1>0$ and find a weight $n\in N$ as claimed.

If we are in situation (2), let $\omega$, $n$ be elements of $\Omega, N$ respectively. Passing to the action obtained by restriction with $\langle \omega, n\rangle$, we obtain again $m_{red}=2$. Using that $\mu(\omega)=1$, we find another weight $\omega_2\in \Omega\cap\langle \omega, n\rangle$ as claimed.
\end{proof}

\begin{Lemma}\label{omegaspans}
Assume we are in the situation of Lemma \ref{graphlemma} and $M$ is not a rational sphere. 

Then, for every fixed point component $F_i$, the set of weights on edges leading from $F_i$ to other components spans the whole weight space $\mathfrak{t}^*$
\end{Lemma}
\begin{proof}
Since the action is effective, the weights of the isotropy representation at any $p\in F_i$ span the whole weight space (see \cite[Remarks before Thm 5]{skjel78}).
If $F_i$ is an isolated fixed point, every such weight appears at an edge to another component and we are done.

If $F_i$ is of type $S^2$, we split the weights of the isotropy representation into two sets $\Omega$, $N$, where $\Omega$ contains those weights that are on edges to other fixed point components, and $N$ are the weights on edges from $F_i$ to itself, except zero. By the above, we have $\mathfrak{t}^*=\langle \Omega \cup N \rangle$. In order to show $\mathfrak{t}^*=\langle \Omega\rangle$, we prove $N\subset \langle \Omega\rangle$. 

By Lemma \ref{spheregraph}, there is another vertex $F_j$ in the graph. Pick a weight $\omega$ on the edge from $F_i$ to $F_j$. If $N$ is empty, we have nothing to show. Pick $n\in N$. By Lemma \ref{s2comb} there is a weight $\omega_2 \ne \omega$ between $F_i$ and $F_j$ such that $\omega_2 \subset \langle \omega, n\rangle$. But then $n\in \langle \omega, \omega_2 \rangle \subset \langle \Omega \rangle$.  
\end{proof}

\begin{Cor}\label{s2cor} Assume we are in the situation of Lemma \ref{graphlemma} and $M$ is not a rational sphere, and that there is a fixed point component $F_i$ of type $S^2$. Then
\begin{itemize}
\item For every other component $F_j$, the weights $\Omega_{ij}$ between $F_i$ and $F_j$ satisfy $\dim\langle \Omega_{ij}\rangle \le 3$.
\item If $d\ge 4$, not all weight sets $\Omega_{ij}$ are identical.
\end{itemize}
\end{Cor}
\begin{proof}
(1) is a straightforward application of a theorem of Sten Hansen to the combinatorial contraint given by Lemma \ref{s2comb}. Two sets like this can only exist if the projective space has dimension at most 2; this is Theorem 8 in \cite{skjel78}. If one is interested in the details of the proof, we will actually prove a strengthened version of this in the proof of Lemma \ref{atmostfour}.

(2) is obtained from (1) and Lemma \ref{omegaspans}.
\end{proof}

\begin{Lemma}\label{lindep}
If $T^d$ acts effectively on a rational $\s{}, \cp$, or $\hp$, then for every triangle with weight sets $A_{12},A_{13},A_{23}$ the polynomials $\Pi_{a\in A_{ij}}a \in H^*(BT)$ are linearly dependent. 
\end{Lemma}
\begin{proof}
If the model is a sphere, there are no triangles and there is nothing to show. 

If the model is a complex projective space, each weight set consists only of one weight by \cite[Lemma 5.2a]{kww}.

If the model $M$ is a quaternionic projective space, choose points $p_i\in F_i$ and an element $x_T$ in $H^4_T(M)$ lifting a generator of $H^4(M, \Q)$. Let $x_i$ denote the restriction of $x_T$ to $H^4_T(p_i) = H^4(BT)$. Then we have $x_i-x_j=\lambda_{ij}\Pi_{a\in A_{ij}}a$ for some nonzero $\lambda_{ij}$ (cf \cite[Lemma 4.2]{kww}). Summing we get the desired linear dependency.
\end{proof}

\begin{Lemma}\label{wiemeler}
There is no $T^2$-action on a rational $\hp^{k\ge 2}$ with isolated fixed points such that all rational weight sets consist of the same two linearly independent elements (up to multiplicity).
\end{Lemma}
\begin{proof}This proof is due to M. Wiemeler, using a trick found in \cite{Li2012}.

Let $a,b\in \Z^{2}$ be primitive representatives of the weights. Then at each fixed point $p_i$ the integral weights are given by $r_1^i a, ..., r_k^i a$, $s_1^i b, ..., s_l^i b$ for some $r_*^*,s_*^*\in \Z$. Since the action is effective we have $T^2 = \ker(a) \oplus \ker(b)$, and by our knowledge of the graph there is a component $N$ of the fixed point set of $ker(a)$ containing all $p_i$. The isotropy representation of $ker(a)$ at $\Nu_{p_i}N$ now has weights $s_1^i b, ..., s_l^i b$. Since $N$ is connected, this representation is the same for all $p_i$, and we see $\{s_1^i, ..., s_l^i\}$ is independent of $i$. Using the same argument with $a$ and $b$ reversed, we see that we have the same $T^2$-representation at each $T_{p_i}M$.

However, there cannot even be an $S^1$-action with isolated fixed points such that all isotropy actions are the same:

If there was such a $S^1 \acton M$, let $x\in H^4(M, \Q)$ denote a generator and $\tilde{x}\in H^4_{S^1}(M, \Q)$ a lift of $x$ to equivariant cohomology. Then, for every $a\in \Q$, $\tilde{x}+a t^2$ is also a lift of $x$, where $t\in H^2(BS^1, \Q)$ is a generator, and by a localization formula going back to Atiyah and Bott \cite{atiyahbott}, we have for each $a$ that 
\begin{align*}
	0\ne \langle x^k, [M] \rangle =& \sum\limits_i \frac{(\tilde{x} + a t^2)^k}{(\prod\limits_j \nu_{ij})t^{2k}}[p_i]\\
=& \Big(\frac{1}{\prod \nu_j} \sum\limits_i (w_i+a)^k\Big)[p_i].
\end{align*}
Here $\tilde{x}|_{p_i} = w_i t^2$ and $\nu_{ij}t$ are the local weights at $p_i$. We used, going to the second line, that these do not depend on $i$. Comparing coefficients in $a^k$, we see that there are the same number of $p_i$ with negative and positive orientation. We split the index set into $I^-$ and $I^+$ accordingly and get 
\begin{align*}
\langle x^k, [M] \rangle =&\frac{1}{\prod \nu_j}\big(\sum\limits_{I^+}(w_i+a)^k - \sum\limits_{I^-}(w_i+a)^k\big)\\
=&\frac{1}{\prod \nu_j} \sum\limits_l \begin{pmatrix}k\\l\end{pmatrix} a^{k-l} \big(\sum\limits_{I^+}w_i^l - \sum\limits_{I^-}w_i^l\big).
\end{align*}
Since $\langle x^k, [M] \rangle$ does not depend on $a$, we see that $\sum\limits_{I^+}w_i^l - \sum\limits_{I^-}w_i^l$ must vanish for $l<k$. However, since $\langle x^k, [M] \rangle$ is not zero, it must not vanish for $l=k$. However, the space of symmetric polynomials on $|I^+|=|I^-|$ variables $x_j$ is generated by $\sum x_j^l$ for $l\le |I^{\pm}|=\frac{k+1}{2}<k$, giving us the desired contradiction.
\end{proof}

\begin{Cor}\label{smallintersection}Assume we are in the situation of Lemma \ref{graphlemma} and $M$ is not a rational sphere.
Then not all weight sets $\Omega_{ij}$ of the graph are identical.
\end{Cor}
\begin{proof}If all fixed point components are isolated fixed points, restriction to any codimension two subtorus will contradict Lemma \ref{wiemeler}. If there is a fixed point component of type $S^2$, this is Corollary \ref{s2cor}.
\end{proof}

With these Lemmas proven, we have gathered enough structure on the graph and its labels to continue the rest of the proof purely combinatorially. The main structural Lemma we will use to recover the condition $m\le 2$ is the following:

\begin{Lemma}\label{transversalweights}Assume we are in the situation of Lemma \ref{graphlemma} and $M$ is not a rational sphere.\\

If, in some triangle with weight sets $A,B,C$, we have $\langle A\rangle \nsubseteq \langle B \rangle$ and $\langle B\rangle \nsubseteq \langle A \rangle$, then the weight sets are, in a suitable basis and up to scaling, either given by $(A, B, C)=(\{a\},\{b\},\{a+b\})$ or $(\{a,a+c\},\{b,b+c\},\{a-b, a+b+c\})$ for linearly independent $a,b$ and where $c$ may be either linearly independent of $\{a,b\}$ or zero.

In particular, $m \le 2$. 
\end{Lemma}
\begin{proof}\,\\
First, we may assume that $A, B, C$ are disjoint. Otherwise, after noting that the intersection of any two sets is the same as the intersection of all three, pick $s$ in the intersection and $a\in A$. Then $\langle s, a\rangle$ has nonempty intersection with $B$ since it contains $s$, and, by counting elements, must contain at least one other element of $B$. Since $\langle s,a\rangle$ is only two-dimensional, that means $\langle s, a\rangle$ is generated by its intersection with $B$, i.e. $a\in \langle B\rangle$ for all $a\in A$, contradicting our assumptions.\\

Pick now $a\in A\backslash \langle B \rangle$, $b\in B\backslash \langle A \rangle$ and split $\langle A\cup B \cup C \rangle$ into a direct sum $\langle a\rangle \oplus \langle b\rangle \oplus \langle A\rangle \cap\langle B\rangle$. To see that this is indeed the decomposition, note that $\langle B \cup a \rangle$ contains all of $A$ and $C$ by Lemma \ref{graphlemma}(2).

Now any element $\omega\in \langle A\cup B \cup C \rangle$ can be written as a sum $\omega = \omega|_a a + \omega|_b b + R$ for some scalars $\omega|_a, \omega|_b$ defined by this equation and some $R\in \langle A\rangle \cap\langle B\rangle$.\\

Written with respect to this decomposition, we have $a_i|_b=b_i|_a=0$ and $a_i|_a, b_i|_b\ne 0$ for all $i$. The first is clear from the definition, for the second note again that $\langle A\cup B \cup C \rangle$ is spanned by any choice of one full weight set and one weight from another weight set. If then $a_i|_a=0$, $\langle a_i, B\rangle$ does not contain $a$, a contradiction, and similarly if $b_i|_b=0$. For the same reason, $c_i|_a\ne 0, c_i|_b\ne 0$ for all $i$.

Order the elements of $A, B, C$ such that $a_1=a$, $b_1=b$. Scale now $a_i, c_i$ such that $a_i|_a=c_i|_a=1$.
Order $C$ such that $c_i \in \langle a, b_i\rangle$, i.e. $c_i = a + \gamma_i b_i$. Furthermore scale $b_i$ such that $b_i|_b=1$ for all $i$.

Assume now there is some $b_i=b+R_i$ with $R_i\ne 0$. The span of $b_i$ and $c_1$ gives an element of $A$ of the form $a - \gamma_1 R_i$, the span of this element and $c_i$ gives an element of $B$ of the form $b + (1 + \frac{\gamma_1}{\gamma_i})R_i$.

Iterating this construction gives elements in $B$ of the form $b + (1+ \frac{\gamma_1 }{\gamma_i} + \dots + (\frac{\gamma_1}{\gamma_i})^k)R_i$. Since there are only finitely many elements of $B$, this implies $\frac{\gamma_1}{\gamma_i}=0$ or a root of unity except 1, i.e. $\frac{\gamma_1}{\gamma_i}\in \{-1,0\}$. It cannot be 0 since $\gamma_1=c_1|_b\ne 0$, so $\gamma_i = -\gamma_1$ for all $i$ with $R_i\ne 0$. 

Scaling $b$ by $-\gamma_1$ and then again scaling $b_i$ such that $b_1=b, b_i|_b=1$ for $i>1$, we get $\gamma_i=1$ whenever $R_i\ne 0$, $c_1=a-b$.\\

If $\mu(a)=2$, i.e. there is a second edge with $R_i=0$, say $R_2$, repeat the above construction replacing $c_1$ with $c_2$ to find $c_2=c_1$ if there is some $R_i\ne 0$. \\

In total now we find the same $R_i$ in $A, B, C$: If $b_i = b + R_i$, by the previous paragraph we have $c_i = a + b + R_i$ and $\langle c_i, b_1\rangle$ gives us an $a_* = a + R_i$; reorder $A$ such that this is $a_i$.

Assume now there are some $R_i \ne R_j$. If $R_i \ne -R_j$, then in the span of $a_i$ and $b_j$ there is an element in $C$ of the form $a + b + R_i + R_j$, since after the last rescaling the only possible values for $\gamma_i$ are $\{+1, -1\}$ and $-1$ only appears at $c_k$ where $R_k=0$ (this appears in $\langle a_i, b_j\rangle$ when $R_i=R_j$). But this means $R_i+R_j$ appears as an $R_k$ somewhere, iteratively producing infinitely many $R_i$, a contradiction.

If $R_i = -R_j$, then in $\langle b_j, c_i \rangle$ we find some $a_*=a+2R_i$. From here we can go to the above case.\\

Now we look at all the possibilities:\\
If all $R_i$ are zero and $\mu(a)=1$, the weights can be scaled to $(\{a\}, \{b\}, \{a+b\})$.\\

If all $R_i$ are zero and $\mu(a)=2$, the weights are $(\{a,a\},\{b,b\},\{a+\gamma_1 b, a+\gamma_2 b\})$.

Using that $a^2, b^2$ and $(a+\gamma_1 b)(a+\gamma_2 b)$ are linearly dependent by Lemma \ref{lindep}, we see that $\gamma_1=-\gamma_2$ and we can scale such that $\gamma_1=-1, \gamma_2=1$.\\

If there is a nonzero $R_i$ and $\mu(a)=2$, we have by the above that $c_1= c_2=a-b$. Restricting to $\langle a,b \rangle$ gives an action with weights $(\{a,a\},\{b,b\},\{a-b, a-b\})$, contradicting Lemma \ref{lindep}.\\

We are left with there being exactly one nonzero $R_i=R$ and $\mu(a)=1$. This means the weights are $(\{a,a+R\},\{b,b+R\},\{a-b, a+b+R\})$, as claimed.
\end{proof}

With this in hand, we can go on to the main structural Lemmas:\\

\begin{Lemma}\label{atmostfour}Assume we are in the situation of Lemma \ref{graphlemma} and $M$ is not a rational sphere.

Then any weight set $A$ satisfies $dim\langle A \rangle \le 4$.
\end{Lemma}

For the proof of this, we use some tools from real projective combinatorics:

\begin{theorem*}[Sylvester-Gallai, '44]\,\\
Let $\Omega$ be a finite subset of some real projective space.\\
Then either there is a line containing exactly two points of $\Omega$, or all points in $\Omega$ are collinear.
\end{theorem*}

as well as its higher-dimensional generalization

\begin{theorem*}[Sten Hansen]\cite{sh}\,\\
Let $\Omega$ be a finite point set in $d$-dimensional real projective space $\mathbb{P}^d$ which is not contained in a hyperplane. Then among the hyperplanes $\mathbb{P}^{d-1}$ determined by points of $\Omega$ there is at least one such that the points of $\Omega$ which it contains, except exactly one, are contained in a $(d-2)$-dimensional subspace $\mathbb{P}^{d-2}$.
\end{theorem*}

Using these Theorems for proving results about weight sets of rational torus actions is something that goes back at least to work of Skjelbred, cf the following proof with \cite[Theorem 8]{skjel78}.

\begin{proof}[Proof of Lemma \ref{atmostfour}\,]\,

Pick any triangle where $A$ is the set of weights at one edge. Let $B, C$ denote the weight sets at the other two edges. By Cor \ref{smallintersection} we may assume $A, B, C$ are not all the same.

If $T:=A \cap B = A \cap B \cap C$ is nonempty, pick $t\in T$ and pass to $\mathbb{P}^k = \mathbb{P}(T_t\mathbb{P}(\langle A\cup B\cup C\rangle))$.

 Here $k = \dim\mathbb{P}(\langle A\cup B\cup C\rangle)-1\ge \dim\langle A\rangle -2$.\\
 Denote by $\Omega$ the image of $(A\cup B \cup C)\backslash T$ in $\mathbb{P}^k$. Then $\Omega$ spans $\mathbb{P}^k$: Pick $s\in T, a \in A\backslash T$. Then $\langle a, s \rangle$ contains an element $b\in B\backslash T$ by counting and we see $s\in \langle a, b \rangle$, so $\langle (A\cup B \cup C)\backslash T\rangle = \langle A\cup B \cup C\rangle$.
 
 Set $\tilde{A}:=A\backslash T$, define $\tilde{B}, \tilde{C}$ likewise.
 
 We may lift each $\omega\in \Omega$ to each of the three sets $\tilde{A}, \tilde{B}, \tilde{C}$: Let $a\in \tilde{A}$ be a lift of $\omega$. Then $\langle t, a\rangle$ has nonempty intersection with $B$ and $C$ since they contain $t$, so the intersection with $B, C$ must contain the same number of elements as the intersection with $A$, which is more than the number of elements in $\langle t,a\rangle \cap T$. These additional elements are lifts of $\omega$ in the sets $\tilde{B}$, $\tilde{C}$.
 
 Pick $\omega_1 \ne \omega_2 \in \Omega$. Pick now a lift $a_1\in \tilde{A}$ of $\omega_1$ and a lift $b_2 \in \tilde{B}$ of $\omega_2$. in $\langle a_1, b_2\rangle$ there is an element $c_3\in \tilde{C}$. The corresponding $\omega_3\in \Omega$ is an element on the line through $\omega_1$ and $\omega_2$ in $\Omega$, which can be neither $\omega_1$ nor $\omega_2$, since by construction $\tilde{A}, \tilde{B}$ and $\tilde{C}$ don't intersect. Since $\omega_{1,2}$ were arbitrary, by the Theorem of Sylvester-Gallai $k\le 1$, i.e. in this case $\dim\langle A\rangle \le 3$.\\
 
 Assume now that $A, B, C$ are disjoint. 
 
 Pick $a\in A$ and move to $\mathbb{P}^k = \mathbb{P}(T_a\mathbb{P}(\langle A\cup B\cup C\rangle))$.\\
 Let the projection be called $\pi$ and define the sets $\Omega := \pi(B\cup C), N:=\pi(A\backslash \{a\})\backslash \Omega$.\\
 Once more, by construction, $\Omega$ generates $\mathbb{P}^k$. We claim the following:
 \begin{itemize}
 \item On any line through two points of $\Omega$ lies a third point of $\Omega \cup N$.
 \item On any line through a point of $\Omega$ and a point of $N$ lies another point of $\Omega$.
 \end{itemize}
For the first, lift the two points of $\Omega$ to $B$ and $C$ respectively. Their span contains a point of $A$, whose projection is the desired point of $\Omega \cap N$.

For the second, lift the point of $\Omega$ to $B$, the point of $N$ to $A$. Their span contains a point of $C$, whose projection is the desired point of $\Omega$.\\

Now apply Sten Hansen's Theorem \cite{sh} to the set $\Omega$ to get a codimension 1 subspace $\mathbb{P}^{k-1} $ and a codimension 2 subset $\mathbb{P}^{k-2}$ such that $\Omega_i := \Omega \cap \mathbb{P}^{k-i}$ satisfy $\Omega_1 = \Omega_2 \dot{\cup} \{\omega\}$ and where $\Omega_i$ generates $\mathbb{P}^{k-i}.$

Denote by $xy$ the line through two points $x, y$. If now $\Omega_2$ contains more than one element, say two different points $\omega_1 \ne \omega_2$, pick a point $\nu \in \omega_1\omega \cap (\Omega \cup N)$. By construction of $\Omega_1$ we have $\nu \in N$. Now there must be a point $\omega_3 \in \omega_2 \nu \cap \Omega$, but $\omega \notin \omega_2 \nu$ and $\omega_2 \nu \cap \Omega_2 = \{\omega_2\}$, a contradiction to $\omega_3 \ne \omega_2$. 

Therefore $\mathbb{P}^{k-2}$ can contain at most one point, i.e. $\dim\mathbb{P}^{k-2} \le 0$, so  $\dim\langle A\rangle \le k+2 \le 4$.
\end{proof}

\begin{Lemma}\label{smallorbig}Assume we are in the situation of Lemma \ref{graphlemma} and $M$ is not a rational sphere.

Then all weight sets $A$ satisfy $\dim\langle A \rangle\in \{1,2,d-1,d\}$. If $d>3$, either all are in $\{1,2\}$ or all are in $\{d-1,d\}$. Furthermore, if $d>3$ and there is a weight set of dimension $d-1$ at a component $F$, then there is also one of dimension $d$.
\end{Lemma}
\begin{proof}\,\\
Choose an edge $(F_A, F_{A}')$ whose weight set $A$ has $3\le \dim\langle A\rangle\le d-1$.

By Lemma \ref{omegaspans}, there is another weight $\omega$ at $F_A$ not in $\langle A\rangle$. By Lemma \ref{transversalweights}, we have $\langle A \rangle \subset \langle B \rangle$, where $B$ is a weight set at $F_A$ containing $\omega$. Since we have $B \subset \langle A\cup \omega\rangle$, we get $\dim\langle B \rangle=\dim\langle A \rangle +1$, proving the last claim. If now $\dim\langle A\rangle\le d-2$, pick another weight $\omega_2$ at $F_A$ not in $\langle A\cup \omega \rangle$. By the same reasoning, we must have $\langle A \rangle \subset \langle C \rangle$ for the corresponding weight set $C$ containing $\omega_2$, and $\dim\langle C \rangle=\dim\langle A\rangle +1$.

Now the triangle spanned by $B$ and $C$ satisfies the assumptions of Lemma \ref{transversalweights}, so $m\le 2$. This is a contradiction since $A$ has at least $\dim\langle A\rangle \ge 3$ elements.\\

Assume now $d>3$ and that there are two weight sets $A, B$ at edges $(F_A, F_{A}'), (F_B, F_{B}')$ with $\dim\langle A \rangle\le 2$ and $\dim\langle B \rangle\ge d-1$. 

As above, if $\dim\langle B \rangle = d-1$, we can find another weight set at $F_B$ where this dimension is $d$ and replace $B$; we assume $\dim\langle B \rangle = d$ going forward.

Consider now the weight set $C$ at $(F_A, F_B)$. If $\dim\langle C \rangle\le 2$, we obtain a contradiction at $F_B$, since $\dim\langle B \rangle$ and $\dim\langle C \rangle$ differ by at least two by construction. Similarly, if $\dim\langle C \rangle\ge d-1$, we find another weight set $C'$ at $F_A$ where this dimension is $d$ and get a contradiction at $F_A$ using $C'$ and $A$.    
\end{proof}

With these Lemmas in hand, the proof of Theorem \ref{maincomb} is straightforward:

\begin{proof}[Proof of Theorem \ref{maincomb}]\,\\
Since $d\ge 5$, Lemmas \ref{atmostfour} and \ref{smallorbig} combine to show $\dim(\langle A \rangle)\le 2$ for any set of weights $A$. Pick a weight set $A$ at some $(F_A, F_{A}')$ and $\omega_1, \omega_2$ at $F_A$ with $\omega_1\notin \langle A \rangle, \omega_2 \notin \langle \omega_1 \cup A \rangle$. Then the weight sets containing $\omega_1$ and $\omega_2$ span a triangle satisfying the condition of Lemma \ref{transversalweights}, so $m\le2$ and the claim follows by Lemma \ref{graphlemma}(5) (i.e.\cite[Lemmas 5.3-5.5]{kww}).
\end{proof}

\begin{Rem}[On the case d=4]\label{t4remark}\,\\
If $d=4$ in Theorem \ref{maincomb}, Lemma \ref{smallorbig} still applies. If there is a fixed point component of type $S^2$, we recover that $M$ is standard by use of Corollary \ref{s2cor}. If all fixed points are isolated and we are in the first case of Lemma \ref{smallorbig}, i.e. all weight sets have dimension at most two, we may proceed as before to recover $M$ as a complex or quaternionic projective space. Otherwise, we must have a weight set with $4$-dimensional span, in particular $m\ge 4$. If we let $k\ge 3$ denote the number of fixed points, recovery of the Cayley Plane amounts to showing that in fact $m=4, k=3$ is the only possible configuration. While this is an open problem in general, it follows for example from the additional assumption that there is a rank one symmetric space of the same dimension and Euler characteristic as $M$, as one might check by examination ($M$ has Euler characteristic $k$ and dimension $2m(k-1)$). Using this, we are able to recover the second case of \cite[Theorem C]{kww} in the case $d=4$.
\end{Rem}

\newpage
\section{Appendix A}
\renewcommand*{\arraystretch}{0.8}
In this section our goal is to give an elementary proof of

\begin{Lemma}\label{polylemma}
Let $c_k$ be the mod $p$ reduction of the Chern class in degree $2k$, \\$k=\lambda p^i, p\nmid \lambda>p$. \\

Then $c_k$ is, modulo products of lower degree Chern classes, nonzero multiple of $P^{p^i}(c_{(\lambda-p+1)p^i})$ if $p$ is odd, or of \\
$P^{2^i}(c_{(\lambda-1)2^i}), \lambda = 3 \mod 4$ resp.
$P^{2^{i+1}}(c_{(\lambda-3)2^i}), \lambda = 1 \mod 4$, if $p=2$.
\end{Lemma}

This implies directly 

\begin{cor*}
The cohomology ring $H^*(BU, \Zp)$ is generated, as $\mathcal{A}_p$-algebra, in degrees $\{2\lambda p^k | k\ge 0, \lambda\le p-1\}$.

In particular, a generating system is given by the mod $p$ reduction of the Chern classes $\{c_{\lambda p^k} | k\ge 0, \lambda \le p-1\}$.
\end{cor*}

Some preliminaries:

Chern classes of complex vector bundles over arbitrary manifolds are defined as pullbacks of the universal Chern classes $c_k$ in the cohomology of the classifying space $BU(k)$, where $k$ is the rank of the vector bundle, by a classifying map $M \to BU(k)$ of the vector bundle.

In order to prove the claim for arbitrary vector bundles, it is therefore enough to verify it in the cohomology ring of $BU(k)$. To simplify things further, there are maps, induced by the inclusion $U(k) \into U(k+1)$, from $BU(k)$ to $BU(k+1)$, that are surjective on cohomology. Passing to the (co-)limit, we obtain from the inclusion maps $U(k)\into U$ maps $BU(k)\to BU$ that are surjective on cohomology. It therefore suffices to verify the claim of the Lemma in the cohomology ring $H^*(BU, \Zp)$.

Using the map $\bigoplus U(1) \into U$ gives us a map $H^*(BU, \Zp) \into \Pi H^*(BU(1),\Zp)=\Pi_{i\in\mathbb{N}}\Zp[x_i]$.

This map identifies $H^*(BU, \Zp)$ injectively with the ring of symmetric polynomials on countably many variables $x_i$ in degree $|x_i|=2$, embedded in the ring $\Pi_{i\in\mathbb{N}}\Zp[x_i]$ of bounded degree formal sums of polynomials on those variables.\\

We will work in this framework as it has several nice properties:\\

If we identify $H^*(BU(k))$ with its image via the inclusion $\bigoplus_{i=1, ..., k} U(1)\into U(k)$ in $H^*(B\big(\bigoplus_{i=1, ..., k} U(1)\big))=\Zp[x_1, ..., x_k]$, we recover $H^*(BU(k))$ as the ring of symmetric polynomials on $k$ variables $x_i$ in degree $|x_i|=2$ and find that the restriction map $H^*(BU)\to H^*(BU(k))$ is given by setting all variables except the first $k$ to zero.\\

The $k$-th Chern class $c_k$ is given by the $k$-th elementary symmetric function on the $x_i$.\\

The action of the Steenrod algebra $\mathcal{A}_p$  can be obtained from the relations\\ $P^i(x_k)= \begin{cases}x_k, i=0\\ x_k^p, i=1\\ 0, \text{ otherwise}\end{cases}$
by use of the Cartan formula.\\

In the following, we will need to do some computation involving symmetric polynomials. To this end, we shall require some compact notation for symmetric polynomials and a calculus for determining products of two such polynomials.

\subsection{Notation for symmetric polynomials}
\,\\
For the symmetric polynomial obtained by permuting the variables of $x_1^{n_1}\dots x_r^{n_k}$ such that the coefficient of $x_1^{n_1}\dots x_k^{n_k}$ remains one, we write $(n_1, ..., n_k)$. We write $((n_1)^{r_1}, ..., (n_k)^{r_k})$ for $(n_1, ..., n_1, ..., n_k, ..., n_k)$, where each $n_i$ is repeated $r_i$ times.

\begin{Ex}\,\\
\begin{itemize}
\item $(1)=x+y+z+...$, 
\item $(2)=x^2+y^2+z^2+...$, 
\item$((1)^2)=(1,1)=xy+xz+yz+...$.\\

\item$(1)\cdot(1)=(x+y+z+...)(x+y+z+...)=x^2+y^2+z^2+2xy+2yz+2xz+...=(x^2+y^2+z^2+...)+2(xy+xz+yz+...)
=(2)+2(1,1)$.\\

\item $c_k = x_1\cdot ... \cdot x_k + ... = (1, ..., 1)=((1)^k)$.

\end{itemize}
\end{Ex}

When using notation like $((n_1)^{r_1}, ..., (n_k)^{r_k})$, we shall always assume $n_i > n_{i+1}$. When using notation of the form $(n_1, ..., n_k)$, we will assume only $n_i \ge n_{i+1}$.\\

It is clear that polynomials of the form $(n_1, ..., n_k)$ give an additive basis of the space of symmetric polynomials. We put an order on this basis lexicographically from the left, i.e. $(n_1, ..., n_k)\le (m_1, ..., m_l)$ if $n_1 < m_1$ or $n_1=m_1$ and $(n_2, ..., n_k)\le (m_2, ..., m_l)$. Here all strings should be considered extended to the right infinitely by zeroes.

\subsection{Computation of Products in $H^*(BU)$}\,\\

In this section we want to describe a way to calculate products of symmetric polynomials in a reasonably compact way. Doing this is, in practice, usually labor-intensive at least for complicated polynomials. However, the results of this section will be enough to let us prove some structural results regarding multiplication in $H^*(BU, \Zp)$ that will suffice for our purposes.\\

Let $p_n=(n_1, ..., n_k), p_m=(m_1, ..., m_r)$.

Multiplying these out, one sees immediately that possible contributions in the product come in the form of $x_{\iota_1}^{n_1}...x_{\iota_k}^{n_k}x_{\sigma_1}^{m_1}... x_{\sigma_r}^{m_r}$, where the $\iota$ and likewise the $\sigma$ are pairwise distinct, but possibly with pairs $\iota_i=\sigma_j$. Since the result of $p_n p_m$ is necessarily a symmetric polynomial, we can restrict our attention to terms where the set of occurring $x_i$ is $\{1, ..., i_{max}\}$. We then write
\begin{align*}
x_{\iota_1}^{n_1}...x_{\iota_k}^{n_k}x_{\sigma_1}^{m_1}... x_{\sigma_r}^{m_r} \cong \begin{bmatrix}
a_{kj}
\end{bmatrix}_{k\le 2, j\le\max(\sigma_*, \iota_*)},
\end{align*}
where $a_{1j}=n_i$ if $\iota_i=j$ and $0$ otherwise, and $a_{2j}=m_i$ if $\sigma_i=j$ and $0$ otherwise. 

The notation should become clear when looking at an example: In $(2,1)\cdot(1)$, there is a contribution $x_1^2 x_2 \cdot x_2$, which we would write as $\begin{bmatrix}
2 & 1\\ 0 & 1
\end{bmatrix}$, since $x_1$ comes with exponent 2 from $(2,1)$ and 0 from $(1)$ and $x_2$ comes with exponent 1 from each. Here the other contributions are
\begin{align*}
x_1^2 x_2 \cdot x_3 \cong \begin{bmatrix} 2&1&0\\0&0&1\end{bmatrix}, x_1^2 x_2 \cdot x_1 \cong \begin{bmatrix}2&1\\1&0\end{bmatrix},
 \end{align*}
and terms obtained by permuting variables, like $x_1 x_2^2 \cdot x_1 \cong \begin{bmatrix}1&2\\1&0\end{bmatrix}$. Note that permutation of variables on the left corresponds to permutations of columns on the right.\\

We shall call contributing terms written in this form \textit{matchings}. Our work will be in enumerating all matchings, or more precisely equivalence classes of matchings, that contribute to the product polynomial, and to determine how many matchings fall in a given class, where we say two matchings are in the same class if they are the same up to permutation of variables. For another example, in the product $(1)\cdot(1)$ we want there to be two contributing classes $x_1 \cdot x_1 \cong \begin{bmatrix}1\\1\end{bmatrix}$ and $x_1 \cdot x_2 \cong \begin{bmatrix}1 & 0 \\ 0 & 1\end{bmatrix}$, and, since we know $(1)(1)=(2)+2(1,1)$, we want a way to obtain the coefficients $1, 2$ from $\begin{bmatrix}1\\1\end{bmatrix}$ and $\begin{bmatrix}1 & 0\\ 0&1\end{bmatrix}$.
Note here that the matching class of $\begin{bmatrix}1\\1\end{bmatrix} \cong x_1\cdot x_1$ produces in the product (up to coefficient) a term of type $x_1^2 + ... =(2)$, while the matching $\begin{bmatrix}1&0\\0&1\end{bmatrix} \cong x_1 \cdot x_2$ produces one of type $x_1 x_2+... =(1,1)$. This is clear from the definitions: A matching $\begin{bmatrix}n_{a_i}\\ m_{b_i}\end{bmatrix}_{i}$ produces a term of type $(n_{a_1}+m_{b_1}, n_{a_2}+m_{b_2}, ...)$, where we take the convention that $n_i=0$ for $i>k$, $a_i, b_i$ are permutations of $\mathbb{N}$ and we ignore columns $\begin{bmatrix}0\\0\end{bmatrix}$ in matching notation.

As we want our polynomials to be sorted in $(o_1, ... )$-notation, a natural choice for a canonical representative of a given matching class should satisfy $n_{a_i}+m_{b_i}\ge n_{a_{i+1}}+m_{b_{i+1}}$. To finally completely fix a permutation of the columns, we shall require $n_i \ge n_{i+1}$ within each block where $n_{a_i}+m_{b_i}$ is constant. It is clear that this defines a unique permutation of columns (where we do not distinguish equal columns), thereby giving a canonical representative in matching notation of each class.\\

Now, let us continue with determining the total coefficient of $(\hat{o}_i=n_{a_i}+m_{b_i})_i$ coming from all matchings with a shared canonical form $\begin{bmatrix}n_{a_i}\\m_{b_i}\end{bmatrix}_i$. It is clear from the definition that the coefficient of $(\hat{o}_i)_i$ in the product is the same as the coefficient of $x_1^{\hat{o}_1}x_2^{\hat{o}_2}\cdot ...$.

Therefore we need to count the number of matchings in our given class with $n_{\tilde{a}_i}+m_{\tilde{b}_i}=\hat{o}_i$. These are the matchings that have the same blocks (sets of columns with constant $n_{\tilde{a}_i}+m_{\tilde{b}_i}=\hat{o}_i$) as the canonical form, but where the columns may be permuted within each block. Let $r_{i,j}$ be the number of columns equal to $\begin{bmatrix}i\\j\end{bmatrix}$ and $(\hat{o}_i)_i=(o_i^{r_i})_i$, where now the $o_i$ are distinct. Then the number of such permutations for the block with sum $o_i$ is $\begin{pmatrix}
r_i \\ (r_{a,b} | a+b=o_i)
\end{pmatrix}$ where we use this notation to denote the multinomial coefficient $\begin{pmatrix}
n \\ (k_1, ..., k_l)
\end{pmatrix}=\frac{n!}{k_1!...k_l!}$, where $\sum k_i=n$. Here the perhaps more familiar case of binomial coefficients is included as $\begin{pmatrix}
n \\ k
\end{pmatrix}=\begin{pmatrix}
n \\ (k, n-k)
\end{pmatrix}$.
The total number of permutations contributing is then the product of these numbers for the individual blocks, i.e. $\prod_i \begin{pmatrix}
r_i \\ (r_{a,b} | a+b=o_i)
\end{pmatrix}$.

\begin{Ex}\,\\
\begin{itemize}
\item $(1)\cdot(1)=\begin{bmatrix}1\\1\end{bmatrix}+\begin{bmatrix}1&0\\0&1\end{bmatrix} = \begin{pmatrix}
1\\1 \end{pmatrix} (2) + \begin{pmatrix}
2\\ 1,1 \end{pmatrix} (1,1) = (2) + 2(1,1)$.
\item $(2,1,1)\cdot (2,1)\\=\begin{bmatrix}2&1&1&0&0\\0&0&0&2&1\end{bmatrix}+\begin{bmatrix}2&1&1&0&0\\0&1&0&2&0\end{bmatrix}+\begin{bmatrix}2&1&1&0&0\\0&2&0&1&0\end{bmatrix}+\begin{bmatrix}2&1&1&0&0\\0&2&1&0&0\end{bmatrix}\\+\begin{bmatrix}2&1&1&0&0\\1&0&0&2&0\end{bmatrix}+\begin{bmatrix}2&1&1&0&0\\1&2&0&0&0\end{bmatrix}+\begin{bmatrix}2&1&1&0&0\\2&0&0&1&0\end{bmatrix}+\begin{bmatrix}2&1&1&0&0\\2&1&0&0&0\end{bmatrix}\\
=\begin{bmatrix}2&0&1&1&0\\0&2&0&0&1\end{bmatrix}+\begin{bmatrix}2&1&0&1\\0&1&2&0\end{bmatrix}+\begin{bmatrix}1&2&1&0\\2&0&0&1\end{bmatrix}+\begin{bmatrix}1&2&1\\2&0&1\end{bmatrix}+\begin{bmatrix}2&0&1&1\\1&2&0&0\end{bmatrix}\\+\begin{bmatrix}2&1&1\\1&2&0\end{bmatrix}+\begin{bmatrix}2&1&1&0\\2&0&0&1\end{bmatrix}+\begin{bmatrix}2&1&1\\2&1&0\end{bmatrix}\\
=6(2^2\,1^3)+6(2^3\,1)+2(3\,2\,1^2)+2(3\,2^2)+(3 \,2\, 1^2)+2(3^2\, 1)+3(4\, 1^3)+ (4\,2\,1)\\
=(4\,2\,1) + 3(4\, 1^3) + 2(3^2\, 1) + 2(3\,2^2) + 3(3 \,2\, 1^2) + 6(2^3\,1) + 6(2^2\,1^3)$\\
\end{itemize}
Note here two things: First, different matchings can contribute to the same term in the product, here $(3\,2\,1^2)$. Second, note that while the 'canonical' form of a matching is most convenient to determine coefficients, for the enumeration of all occurring matchings it is much more convenient to instead leave the top row sorted, and break ties by sorting the bottom. Then one can get all occurring matchings by iterating through all partitions of the bottom row into descending subsets like above, and then go to canonical form by sorting and deleting zero columns.\\
\end{Ex}

As a final note in this section, sometimes one is not interested in calculating the full product, but rather only in certain coefficients in the result (e.g. one may only want to determine the coefficient of $(3\,2\,1^2)$ in the above example without calculating all terms. In this case, one needs only to find all matchings that add up to the desired result column-wise, i.e. in the example the columns would have to add up to $(3\,2\,1\,1)$. This tends to be less difficult in the relevant situations, as one can use reasoning to find these matchings, i.e. here one sees quickly that any matching must start either $\begin{bmatrix}
1&2&*\\2&0&*
\end{bmatrix}$ or $\begin{bmatrix}
2&0&*\\1&2&*
\end{bmatrix}$ and find that the contributing terms are $\begin{bmatrix}1&2&1&0\\2&0&0&1\end{bmatrix}+\begin{bmatrix}2&0&1&1\\1&2&0&0\end{bmatrix}=3(3\,2\,1^2)$. Reasoning like this will be used extensively in the following proofs, where usually only single coefficients need to be calculated explicitly.\\

\subsection{Proofs of the results}\,\\
As an application, we prove

\begin{Lemma}
For every $p,l$, $S_l:=\{((n_1)^{r_1},..., (n_k)^{r_k})| \exists i: p^l\nmid r_i\}$ generates additively an ideal of $H^*(BU, \Zp)$.
\end{Lemma}

\begin{proof}
Let $p_n=((n_1)^{r_1},..., (n_k)^{r_k})\notin S_l$, i.e. $p^l|r_i$ for all $i$. \\
Let $p_m=((m_1)^{s_1},..., (m_l)^{s_l})\in S_l$ and consider the coefficient of $p_n$ in the product of $p_m$ with some other basis polynomial $p_o$. Let $i$ be chosen such that $p^l\nmid s_i$.\\
For every contributing matching we consider the columns of the form $\begin{bmatrix}
m_i\\
o_j
\end{bmatrix}$, where $j$ may vary. Of course there are in total $s_i$ of these. In particular, one of these occurs a number of times $r_{i,j}$ not divisible by $p^l$. This implies the coefficient coming from this matching comes with a factor of $\begin{pmatrix}r\\
..., r_{i,j}, ...
\end{pmatrix}$, where $r$ is the multiplicity of $m_i+o_j$ in $p_n$ and the bottom ranges over $r_{a,b}$ with $m_a+o_b=m_i+o_j$. However, since $p^l|r$ by assumption on $p_n$ and $p^l\nmid r_{i,j}$ by construction of $i,j$, this multinomial coefficient is zero mod $p$.

In total we see that every contribution comes with a factor of zero, so $p_n$ does not appear in the product of $p_m$ with any symmetric polynomial.
\end{proof}

\begin{proof}[Proof of Lemma \ref{polylemma}\,]\,\\
We prove the following:
Let $p_n=((n_1)^{r_1},..., (n_k)^{r_k})$ be any polynomial in degree $2\sum n_i r_i=2\hat{\lambda} p^{\hat{l}}, p\nmid \hat{\lambda} > p$. Then $p_n$ is, modulo products of lower degree elements, in the image of $P^{p^{\hat{l}}}$.

Since $H^{|p_n|-2(p-1)p^{\hat{l}}}(BU, \Zp)$ is generated by the Chern class in this degree and products, and $P^*$ of a product decomposes as a sum of products by the Cartan formula, this implies the Lemma as stated.\\

Suppose $p_n=((n_1)^{r_1},..., (n_k)^{r_k})$ is a lexicographically maximal counterexample of minimal degree. We perform a sequence of reductions:\\

If $k>1$, consider $((n_1)^{r_1},..., (n_{k-1})^{r_{k-1}})\cdot ((n_k)^{r_k})=((n_1)^{r_1},..., (n_k)^{r_k})+R$, where all terms in $R$ are lexicographically larger than $p_n$. \\
We may assume $k=1, ((n_1)^{r_1})=:((n)^r)$.\\

Write $r=\lambda p^l + \hat{r}$, where $0 \ne \lambda \le p-1$, $p^{l+1}|\hat{r}$.\\
If $\hat{r}\ne 0$, consider $((n)^{\hat{r}})((n)^{\lambda p^l})=((n)^r)+R$ with terms in $R$ that are lexicographically larger than $((n)^r)$. We may assume $\hat{r}=0$.\\

If $l\ne \hat{l}$, then since $|p_n|=2\lambda p^l n = 2\hat{\lambda}p^{\hat{l}}$ and $p$ does not divide $\lambda$ or $\hat{\lambda}$, we must have $p|n$. \\
In this case $((n)^r)=((\frac{n}{p})^r)^p$ by $\Zp$-linearity of the Frobenius-homomorphism. 

In particular we claim this shows $S_k$ is generated by products for $k<\hat{l}+1$: If $p_n\in S_k$ for such a $k$, one may check to see the product terms above are also in $S_k$, which means so are the terms in the $R$ at each step. Since $l+1=min\{k |\,\, p_n \in S_k\}$, all these terms arrive in the case $l\ne \hat{l}$ by the assumption $k\ne \hat{l}+1$ and are seen to be products.\\

We return to the proof and may now assume $l=\hat{l}$. We reduce to $\lambda=1$ by considering \\
$((n)^{p^l})^{\lambda}=\begin{pmatrix}
\lambda p^l \\ p^l, ..., p^l
\end{pmatrix}((n)^{\lambda p^{l}})+R = \lambda! ((n)^{\lambda p^{l}})+R$  with lexicographically larger $R$ if $\lambda >1$. (Note that $\lambda! \ne 0$ mod $p$ if and only if $\lambda <p$).\\

Since $nr=np^l=\hat{\lambda}p^l$, we may write $n=k_n p + \lambda_n$ for some $k_n\ge 1, 0<\lambda_n \le p-1$.\\
If $\lambda_n \ne p-1$, we consider\\
$P^{p^l}(((n-p+1)^{p^l})) = \underline{(n-p+1)^{p^l}} ((n)^{p^l})=\underline{(n-p+1)} ((n)^{p^l})$, where the first equality holds modulo $S_l$ and underscored terms are supposed to be read as scalars.\\

For the remaining case $n= k_n p + p - 1$ we must consider the cases of odd and even primes separately: \\If $p$ is odd, we have
\begin{align*}
P^{2p^l}((k_n p -(p-1))^{p^l}) = \begin{pmatrix}
p-1 \\ 2
\end{pmatrix} ((n)^{p^l}) + R, \text{ if } k_n\ge 2,\\
P^{p^l}((1)^{p^{l+1}})=((p)^{p^l}(1)^{(p-1)p^l}) = ((n)^{p^l}) + R , \text{ if } k_n=1.
\end{align*}
Here the $R$-terms are in $S_l$ in the first line, and are products or in $S_l$ in the second. Note here that the image of $P^{2p^l}$ lies in the image of $P^{p^l} \mod S_l$ since by the Adem relation for $P^{p^l}\circ P^{p^l}$ we have $P^{2p^l}=\frac{1}{2} P^{p^l}\circ P^{p^l} + \sum \alpha_{a,b}P^a\circ P^b$, where $p^l$ divides none of $a,b$, which implies all terms in the sum have image in $S_l$.\\

For $p=2$, 
\begin{align*}
&\text{if } 2k_n+1 =1 \mod 4, ((2k_n+1)^{2^l}) = P^{2\cdot 2^l}((2k_n-1)^{2^l}),\\
&\text{if } 2k_n+1 =3 \mod 4, ((2k_n+1)^{2^l}) = P^{2^l}((k_n)^{2\cdot 2^l}) + ((k_n)^{2^l})((k_n)^{2^l} (1)^{2^l}) + R,
\end{align*}
where terms in $R$ are in $S_l$.
\end{proof}

\renewcommand*{\arraystretch}{1}

\newpage
\section{Appendix B}
In this section we show how we classified the isotropy weight sets of effective $T^4$-actions without finite isotropy groups (near a chosen fixed point) to arrive at:\\

\begin{table}[h!]
\renewcommand*{\arraystretch}{0.9}
  \begin{center}
    \caption*{Table \ref{tab:t4list}. List of connected isotropy $T^4$ weight sets}
    \begin{tabular}{r|c|l} 
      \textbf{Size} & \textbf{weights not including standard basis} & \textbf{Splitting $T^3$?}\\
      \hline
      4 & $\emptyset$ & \checkmark\\
      5 & $\begin{pmatrix}1\\ 0\\ 0\\ -1\end{pmatrix},\begin{pmatrix}1\\ 0\\ -1\\ -1\end{pmatrix},\begin{pmatrix}1\\ -1\\ -1\\ -1\end{pmatrix}$ & \checkmark\\
      6 & $\begin{pmatrix}1 & 1\\ 0 & 0\\ 0 & -1 \\ -1 & 0\end{pmatrix},\begin{pmatrix}1 & 1\\ 0 & -1\\ 0 & -1 \\ -1 & 0\end{pmatrix},\begin{pmatrix}1 & 1\\ 0 & -1\\ -1 & 0 \\ -1 & -1\end{pmatrix},\begin{pmatrix}1 & 0\\ -1 & 0\\ 0 & 1 \\ 0 & -1\end{pmatrix}$ & all except last\\
      7 & $\begin{pmatrix}1 & 1 & 1\\ 0 & 0 & 0\\ 0 & -1 & -1 \\ -1 & 0 & -1\end{pmatrix}, \begin{pmatrix}1 & 1 & 1\\ 0 & 0 & -1\\ 0 & -1 & 0 \\ -1 & 0 & 0\end{pmatrix}, \begin{pmatrix}1 & 1 & 1\\ 0 & 0 & -1\\ 0 & -1 & 0 \\ -1 & 0 & -1\end{pmatrix}, \begin{pmatrix}1 & 1 & 1\\ 0 & 0 & -1\\ 0 & -1 & -1 \\ -1 & 0 & -1\end{pmatrix}$ & \checkmark\\
      8 & $\begin{pmatrix}1 & 1 & 1 & 1\\ 0 & 0 & 0 & -1\\ 0 & -1 & -1 & 0 \\ -1 & 0 & -1 & 0\end{pmatrix} , \begin{pmatrix}1 & 1 & 1 & 1\\ 0 & 0 & -1 & -1\\ 0 & -1 & 0 & -1 \\ -1 & 0 & -1 & 0\end{pmatrix}$ & \checkmark\\
      9 & $\begin{pmatrix} 1 & 1 & 1 & 1 & 1\\ 0 & 0 & 0 & -1 & -1\\ 0 & -1 & -1 & 0 & 0 \\ -1 & 0 & -1 & 0 & -1\end{pmatrix}, \begin{pmatrix} 1 & 1 & 1 & 1 & 1\\ 0 & 0 & -1 & -1 & -1\\ 0 & -1 & 0 & -1 & -1 \\ -1 & 0 & -1 & 0 & -1\end{pmatrix}$ & \checkmark\\
      10 & $\begin{pmatrix} 0 & 0 & 0 & 1 & 1 & 1\\ 0 & 1 & 1 & 0 & 0 & -1\\ 1 & 0 & -1 & 0 & -1 & 0 \\ -1 & -1 & 0 & -1 & 0 & 0\end{pmatrix}$ & $\times$
    \end{tabular}
  \end{center}
\renewcommand*{\arraystretch}{1}
\end{table}

In order to achieve this, we identify the weight space with $\Z^4$. Change of identification $Hom(T^4, S^1)\equiv \Z^4$ corresponds to basis change on the right side.\\

We call the set of weights of the isotropy representation $\Omega$. The condition of effectivity of the action now corresponds to the span of $\Omega$ having full rank in $\Z^4$. The condition on isotropy groups means that any rank 4 sublattice of $\Z^4$ spanned by elements of $\Omega$ has index 1. In particular, the span of $\Omega$ itself has index 1, so $\Omega$ contains a basis of $\Z^4$. We may assume $\Omega$ contains the standard basis.\\

The above condition on sublattices is equivalent to the condition that for any four $\omega_1, ..., \omega_4 \in \Omega$, $\det(\omega_1, ..., \omega_4)\in \{-1,0,1\}$ ($\star$). This is what we will actually check for later. Note that if $\Omega$ contains the standard basis, this condition immediately implies the entries of $\omega\in \Omega$ can only be $0$ or $\pm 1$.\\

Furthermore one should note that, if $\Omega$ satisfies these conditions, then any set $\bar{\Omega}$ with $\{e_i\} \subset \bar{\Omega} \subset \Omega$ does as well. This implies all $\Omega$ with these properties can be obtained by iteratively adding single elements to the standard basis.\\

We use the following algorithm to classify all such $\Omega$:\\

\begin{enumerate}
\item Assume we have a list $L_k$ of all sets $\Omega$ containing the standard basis satisfying ($\star$) containing exactly $k$ elements, starting at $k=4, L_4=\{\{e_i|1\le i\le 4\}\}$.
\item After putting an order on $V:=\{-1,0,1\}^4\backslash \{\pm e_i, 0\}$, check ($\star$) for $\Omega \cup v$ for each $\Omega \in L_k, v\in \{v\in V, v>\omega  \, \forall \omega \in \Omega\}$. If so, add $\Omega \cup v$ to a new list $\hat{L}_{k+1}$.
\item check the $\hat{\Omega}\in \hat{L}_{k+1}$ for isomorphisms by basis change. If multiple $\hat{\Omega}$ belong to the same isomorphism class, keep only the representative minimal with regard to the lexicographic ordering induced by the chosen order on $V$.
\item The obtained list $L_{k+1}$ will now contain a lexicographically minimal representative of each isomorphism class of sets $\hat{\Omega}$ containing the standard basis and satisfying ($\star$) with exactly $k+1$ elements.
\item Repeat the above until $L_{k}=\emptyset$.
\end{enumerate}

Implementing this algorithm in code, one computationally generates a complete classification of such $\Omega$ agreeing with \ref{tab:t4list}.


\begin{Rem} We also make claims about which actions "contain a splitting $T^3$", i.e. admit a quotient $T^4/S^1$ whose isotropy near $p\in \Fix(S^1)$ splits as a direct sum of three circle actions. The weight set of such a quotient action is given (under suitable identifications) by the intersection of the original weight set with the 3-dimensional subspace of weights that contain the lie-algebra of the $S^1$ in their kernel. Looking for a "splitting $T^3$" now corresponds to finding a 3-dimensional subspace whose intersection with the weight set is exactly three linearly independent weights.\\
In our case, the ordering of $V$ is chosen in such a way that, if there is such a subspace, the lexicographically minimal representative of any isomorphism class will have splitting $\{e_2, e_3, e_4\}$. Checking whether this indeed gives a splitting $T^3$ of course only amounts to verifying that the first coordinate of all other weights is non-zero, which can be done by eye. It only remains to check that the two weight sets where this is not the case really do not admit any splitting $T^3$, which is not difficult.
\end{Rem}

\clearpage 

It is also possible to give Table 1 in a basis that makes all entries positive:\\

\begin{table}[h!]
  \begin{center}
    \caption{List of connected isotropy $T^4$ weight sets in alternate basis}
    \label{tab:t4listalt}
    \begin{tabular}{l|c|r} 
      \textbf{Size} & \textbf{weights not including standard basis} & \textbf{Splitting $T^3$?}\\
      \hline
      4 & $\emptyset$ & \checkmark\\
      5 & $\begin{pmatrix}1\\ 0\\ 0\\ 1\end{pmatrix},\begin{pmatrix}1\\ 0\\ 1\\ 1\end{pmatrix},\begin{pmatrix}1\\ 1\\ 1\\ 1\end{pmatrix}$ & \checkmark\\
      6 & $\begin{pmatrix}1 & 1\\ 0 & 0\\ 0 & 1 \\ 1 & 0\end{pmatrix},\begin{pmatrix}1 & 1\\ 0 & 1\\ 0 & 1 \\ 1 & 0\end{pmatrix},\begin{pmatrix}1 & 1\\ 0 & 1\\ 1 & 0 \\ 1 & 1\end{pmatrix},\begin{pmatrix}1 & 0\\ 1 & 0\\ 0 & 1 \\ 0 & 1\end{pmatrix}$ & all except last\\
      7 & $\begin{pmatrix}1 & 1 & 1\\ 0 & 0 & 0\\ 0 & 1 & 1 \\ 1 & 0 & 1\end{pmatrix}, \begin{pmatrix}1 & 1 & 1\\ 0 & 0 & 1\\ 0 & 1 & 0 \\ 1 & 0 & 0\end{pmatrix}, \begin{pmatrix}1 & 1 & 1\\ 0 & 0 & 1\\ 0 & 1 & 0 \\ 1 & 0 & 1\end{pmatrix}, \begin{pmatrix}1 & 1 & 1\\ 0 & 0 & 1\\ 0 & 1 & 1 \\ 1 & 0 & 1\end{pmatrix}$ & \checkmark\\
      8 & $\begin{pmatrix}1 & 1 & 1 & 1\\ 0 & 0 & 0 & 1\\ 0 & 1 & 1 & 0 \\ 1 & 0 & 1 & 0\end{pmatrix} , \begin{pmatrix}1 & 1 & 1 & 1\\ 0 & 0 & 1 & 1\\ 0 & 1 & 0 & 1 \\ 1 & 0 & 1 & 0\end{pmatrix}$ & \checkmark\\
      9 & $\begin{pmatrix} 1 & 1 & 1 & 1 & 1\\ 0 & 0 & 0 & 1 & 1\\ 0 & 1 & 1 & 0 & 0 \\ 1 & 0 & 1 & 0 & 1\end{pmatrix}, \begin{pmatrix} 1 & 1 & 1 & 1 & 1\\ 0 & 0 & 1 & 1 & 1\\ 0 & 1 & 0 & 1 & 1 \\ 1 & 0 & 1 & 0 & 1\end{pmatrix}$ & \checkmark\\
      10 & $\begin{pmatrix} 1 & 0 & 0 & 1 & 0 & 1\\ 1 & 1 & 0 & 1 & 1 & 1\\ 0 & 1 & 1 & 1 & 1 & 1 \\ 0 & 0 & 1 & 0 & 1 & 1\end{pmatrix}$ & $\times$
    \end{tabular}
  \end{center}
\end{table}

\clearpage

\clearpage

\printbibliography

\end{document}